\documentclass{article}[12pt]
\usepackage{geometry}
\usepackage{graphicx,amsthm,natbib,mathtools,amsfonts}
 \usepackage{mathrsfs}
\usepackage{multirow,mfirstuc}
\usepackage{subfigure}
\usepackage{algorithm}
\usepackage{algorithmic}
\usepackage{cleveref}
\usepackage{xcolor}
\newtheorem{theorem}{Theorem}

\newtheorem{assumption}[theorem]{Assumption}

\newtheorem{definition}[theorem]{Definition}
\newtheorem{lemma}[theorem]{Lemma}

\newcommand{\Acal}{{\cal A}}

\newcommand{\Ical}{{\cal I}}

\newcommand{\Lcal}{{\cal L}}

\newcommand{\Scal}{{\cal S}}

\newcommand{\sign}{\text{sign}}   
  
\begin{document} 
\title{\LARGE An Extrapolated Iteratively Reweighted $\ell_1$ Method 
with    Complexity Analysis}

\author{
{\large
Hao Wang\thanks{Email: haw309@gmail.com}
} \\[2pt]
{\large School of Information Science and Technology, ShanghaiTech
University}
\\[6pt]
{\large
Hao Zeng\thanks{Email: zenghao@shanghaitech.edu.cn}
} \\[2pt]
{\large School of Information Science and Technology, ShanghaiTech
University}
\\[6pt]
{\large  and}\\[6pt]
{\large  Jiashan Wang}\thanks{Email: jsw1119@gmail.com}\\[2pt]
\large Department of Mathematics, University of Washington
}

\date{\large\today}

\maketitle

\begin{abstract}
The iteratively reweighted $\ell_1$ algorithm is a widely used method for solving various regularization problems, 
which generally minimize a differentiable loss function combined with a convex/nonconvex 
regularizer to induce sparsity in the solution. 
However, the convergence and the complexity of iteratively reweighted $\ell_1$ algorithms is generally difficult 
to analyze, especially for non-Lipschitz differentiable regularizers such as $\ell_p$ norm regularization with 
$0<p<1$.  In this paper,   we propose, analyze and test a reweighted $\ell_1$ algorithm combined with 
the extrapolation technique under the assumption of \emph{Kurdyka-\L ojasiewicz} (KL)  property  on the objective. 
Unlike existing iteratively reweighted $\ell_1$ algorithms with extrapolation, our method does not require 
the Lipschitz differentiability on the regularizers nor   the smoothing parameters in the weights 
bounded away from 0.    We show the proposed algorithm converges uniquely 
to a stationary point of the regularization problem and has local linear complexity---a much stronger result than existing 
ones.  Our numerical experiments show the efficiency of our proposed method.
 
\end{abstract}

\section{Introduction}
\label{intro}
Recently, sparse regularization has received an increasing contemporary attentions among researchers due to its various important applications, e.g., compressed sensing, machine learning, and image processing  \cite{figueiredo2007gradient,lustig2007sparse,jaggi2011sparse,mairal2010online,mairal2007sparse,zeyde2010single,luo2017revisit}. The goal of sparse regularization is to find sparse solutions of a mathematical model  such that the  model performance can be better generalized to  future data.  A common approach of this regularization technique is to 
add a  regularizer term  to the objective such that most of the components in the resulted solution are zero. In this paper, 
we focus on the  $\ell_p$-norm regularization optimization problem of the following form
\begin{equation}
\tag{P}\label{prob.lp}
\min_{x\in \mathbb{R}^n} F(x):=f(x)+\lambda \|x\|^p_p,
\end{equation} 
where $f:\mathbb{R}^n\to \mathbb{R}$ is  continuously differentiable, $p\in(0,1)$  and $\lambda > 0$ is the prescribed  regularization parameter. The $\ell_p$ norm is defined as 
$\|x\|_p:= (\sum_{i=1}^n |x_i|)^{1/p}$.  Compared  with the $\ell_1$ regularizer, $\ell_p$ regularizer is often believed 
to be a better approximation to the $\ell_0$ regularizer, i.e., the number of nonzeros of the involved vector. However, due to the nonsmooth and non-Lipschitz differentiable nature  of  the $\ell_p$-norm,  this problem is difficult to handle and analyze. In fact, it has been proven in   \cite{ge2011note} to be  strongly NP-hard.

The iteratively reweighted $\ell_1$ (IRL1)  algorithm  \cite{candes2008enhancing,chen2010convergence,chartrand2008iteratively,yu2019iteratively,lu2014proximal,wang2019relating} has been widely applied to solve regularization problems to induce sparsity in the solutions.  It can easily handle 
  various  regularization terms including  $\ell_p$-norm, log-sum \cite{lobo2007portfolio}, SCAD \cite{fan2001variable}  and MCP \cite{zhang2010nearly} by approximating regularizer with a weighted $\ell_1$ norm in each iteration. 
  For example,  the technique proposed by   Chen  \cite{chen2010convergence} 
and Lai \cite{lai2011unconstrained}   adds  smoothing perturbation $\epsilon$  to each $|x_i|$ to formulate  
the $\epsilon$-approximation of the $\ell_p$ norm. In this case, the objective is replaced by 
\begin{equation}\label{l2-lp-appro1}
f(x)+ \lambda \sum_{i=1}^n \left( |x_i| + \epsilon_i \right)^p, 
\end{equation}
with prescribed  $\epsilon>0$.  At each iteration $x^k$, the 
iteratively reweighted $\ell_1$ method solves the 
subproblem  in which  each $ \left( |x_i| + \epsilon \right)^p$ is replaced by its linearization
\begin{equation}\label{first.l2}
p(|x_i^k|+\epsilon_i)^{p-1}|x_i|.
\end{equation}
In this case, large $\epsilon$ can smooth out many local minimizers, while small values make the subproblems difficult to solve and the 
algorithm easily  trapped into bad local minimizers.
To obtain an accurate approximation of \eqref{prob.lp},  Lu \cite{lu2014iterative} proposed a dynamic updating strategy 
to drive  $\epsilon$ from an initial relatively large value to 0   as $k\to \infty$. 
Recently, Wang et al.  \cite{wang2019relating}   show  the property that the iterates generated by the 
 IRL1 algorithm have local stable sign value. Based on this,  
 they  also present a novel  updating strategy for $\epsilon_i$ that    only drive $\epsilon_i$ associated with the  nonzeros in  the limit point 
  to 0 while keeping others  
 bounded away from 0.

Since  Nesterov \cite{nesterov1983method} first proposed the extrapolation techniques in gradient method, many works  focus  on analyzing and improving the  convergence rate  of the IRL1 algorithms. In this approach,  a linear combination of previous two steps are 
used to update next step.
Nesterov’s extrapolation techniques \cite{nesterov1983method,nesterov1998introductory, nesterov2009primal,nesterov2013gradient} have also been widely applied  
  to accelerate the performance of the first-order methods and convex composite optimization problems, for example \cite{auslender2006interior,becker2011templates,lan2011primal} and \cite{tseng2010approximation}.    During the past decade, it is proven 
  to be a   successful accelerating approach 
  when applied to various algorithms. For example, Amir and Teboulle  \cite{beck2009fast}  presented a fast iterative shrinkage-thresholding algorithm (FISTA) using this technique.  
  As for IRL1 methods,    Yu and   Pang  \cite{yu2019iteratively}   proposed several versions of IRL1 algorithms with extrapolation and analyzed the global convergence.


In view of the success of IRL1 combined with the extrapolation techniques, in this paper we propose and analyze the 
 Extrapolated Proximal Iteratively Reweighed $\ell_1$ method (E-PIRL1) to solve the $\ell_p$-norm regularization problem. 
 We  show the global  convergence and the local complexity   of the proposed methods under the Kurdyka-\L ojasiewicz   (K\L)    property  \cite{boltedanlew17, Bolte2014}; this property is a mild condition  and generally believed to  capture a broad spectrum of the local geometries that a nonconvex function can have and  has been shown to hold ubiquitously for most practical functions. 
 In particular, we show that our method converges to a first-order optimal solution of the $\ell_p$ regularization problem and  local sublinear and linear convergence rates are established---a stronger result than most existing ones. 
 
 The proximal iteratively reweighted $\ell_1$ methods with  extrapolation by 
    Yu and   Pang in  \cite{yu2019iteratively}  can be an immediate and most related  predecessor of our proposed method.  
  The main  differences  of our work  can be summarized as follows. 
  \begin{enumerate}
  \item The algorithms in  \cite{yu2019iteratively}  are designed 
 for solving  problems with regularization term 
$\sum\limits_{i=1}^nr_i(|x_i|)$  where $r_i$ 
is assumed to be smooth on $\mathbb{R}_{++}$, concave and strictly increasing on $\mathbb{R}_+$ with   $r_i(0)=0$. 
However, it is assumed in \cite{yu2019iteratively} that    $\lim_{|x_i| \to0}r_i'(|x_i|)$   exits, meaning 
$r_i$ is Lipschitz differentiable on $\mathbb{R}_{++}$, which is not the case for $\ell_p$-norm regularization term. 

\item The algorithms in  \cite{yu2019iteratively}  can be extended to the  $\ell_p$ norm regularization problem by
 keeping the smoothing parameter $\epsilon$   bounded away from 0. In this case, the algorithms converge
 to the optimal solution of the approximated  $\ell_p$ regularization problem instead of the original 
 one.  In contrast, our algorithm drives $\epsilon\to 0$ over iterations so that 
 the optimal solution of the original problem can be obtained. 
\item The third difference is the local  convergence rate  provided in our work, whereas  no complexity analysis is provided in \cite{yu2019iteratively}; this
 is also the most important contribution of our work.
 \end{enumerate}

\subsection{Notation}
We denote $\mathbb{R}$ and $\mathbb{Q}$ as the set of real numbers and rational numbers. The set $\mathbb{R}^n$ is the real $n$-dimensional 
Euclidean space with $\mathbb{R}^n_+$ being the positive orthant in $\mathbb{R}^n$ and $\mathbb{R}^n_{++}$ the interior of $\mathbb{R}^n_+$. 
In $\mathbb{R}^n$, denote $\|\cdot\|_p$ as the $\ell_p$ norm with $p\in(0,+\infty)$, i.e., 
$\|x\|_p = \left(\sum_{i=1}^n |x_i|^p\right)^{1/p}$.   Note  that for $p\in(0,1)$, 
this  does not define a proper norm due to its lack of subadditivity. 
If function $f: \mathbb{R}^n \to \bar{ \mathbb{R}}:=\mathbb{R} \cup \{+\infty\}$
is convex, then the subdiferential of $f$ at $\bar x$ is given by 
\[ \partial f(\bar x):=\{ z \mid f(\bar x) + \langle z, x-\bar x\rangle \le f(x),  \  \forall x\in\mathbb{R}^n\}.\]
In particular, for $x\in \mathbb{R}^n$,  $\partial \|x\|_1=\{ \xi \in \mathbb{R}^n \mid  \xi_i \in \partial |x_i|, i=1,\ldots, n\}.$
Given a lower semi-continuous function $f$,  the limiting subdifferential at $a\in \text{dom}f$ is defined as 
\[ \bar\partial f(a) :=\{ z^* = \lim_{x^k\to a,  f(x^k)\to f(a)} z^k, \  z^k\in\partial_F f(x^k)\}\]
and the Frechet subdifferential of $f$  at $a$ is defined as 
\[ \partial_F f(a):=\{z\in\mathbb{R}^n \mid \liminf_{x\to a } \frac{f(x)-f(a)-\langle z, x-a\rangle}{\|x-a\|_2} \ge 0\}. \]
The Clarke subdifferential $\partial_c f$ is the convex hull of the limiting subdifferential. 
It holds true that $\partial_F f(a) \subset \bar\partial f(a) \subset \partial_cf(a)$.  
For convex functions, $\partial f(a) = \partial_F f(a) = \bar \partial f(a) = \partial_c f(a)$ 
and for differentiable $f$, $  \partial_F f(a) = \bar \partial f(a) = \partial_c f(a) = \{\nabla f(a)\}$.

For $f: \mathbb{R}^n \to \mathbb{R}$ and index sets $\Acal$ and $\Ical$ satisfying
 $\Acal \cup \Ical  = \{1, \ldots, n\}$,  
let $f(x_{\Acal})$ be  the function in the reduced space $\mathbb{R}^{|\Acal|}$ 
by fixing $x_i = 0, i\in \Ical$.  For $a, b\in\mathbb{R}^n$,  $a\le b$ means the inequality holds for each component, i.e., 
$a_i \le b_i $ for $i=1,\ldots, n$ and $a\circ b$ is the component-wise product of $a$ and $b$, i.e., $(a\circ b)_i = a_ib_i$ for $i=1,..., n$. 
 For a closed convex set $\chi \subset \mathbb{R}^n$, define the Euclidean distance of point $a\in\mathbb{R}^n$
to $\chi$ as $\text{dist}(a,  \chi) = \min_{b\in \chi} \| a - b\|_2$. 
Let $\{-1,0,+1\}^n$ be the set of vectors in $\mathbb{R}^n$ filled with elements in $\{-1,0,+1\}$. 
The support of $x\in\mathbb{R}^n$ is defined as the set 
$\{ i \mid x_i \neq 0, i=1,...,n\}$.

\subsection{Kurdyka-\L ojasiewicz property}\label{sec.KL}

Kurdyka-\L ojasiewicz property is applicable to a wide range of problems such as nonsmooth semi-algebraic minimization problem \citep{Bolte2014}, 
and serves as a basic assumption to guarantee the convergence of many algorithms.  For example, 
  a series of convergence results for gradient descent methods are proved 
  in  \cite{attouch2013convergence}  under the assumption that the objective satisfies the KL property.  
The definition of  Kurdyka-\L ojasiewicz property  is given below.

\begin{definition}[Kurdyka-\L ojasiewicz property]\label{df:KL}
	The function $f:\mathbb{R}^n\to \mathbb{R}\cup\{+\infty\}$ is said to have the Kurdyka-\L ojasiewicz property at $x^*\in \text{dom}\bar\partial f$ if there exists $\eta\in (0,+\infty]$, a neighborhood $U$ of $x^*$ and a continuous concave function $\phi:[0,\eta) \to \mathbb{R}_+$ such that:
	\begin{enumerate}
		\item[(i)] $\phi(0)=0$,
		\item[(ii)] $\phi$ is $C^1$ on $(0,\eta)$,
		\item[(iii)] for all $s\in(0,\eta)$, $\phi'(s)>0$,
		\item[(iv)] for all $x$ in $U\cap [f(x^*)<f<f(x^*)+\eta]$, the Kurdyka-\L ojasiewicz inequality holds
		\[
		\phi'(f(x)-f(x^*))\text{dist}(0, \bar\partial f(x))\geq 1.
		\]
	\end{enumerate}   
\end{definition}
If $f$ is smooth, then condition (iv) reverts to  \cite{attouch2013convergence}
\[\| \nabla (\phi\circ f)(x)\| \ge 1.\]

Of particular interests is the class of \emph{Semialgebraic functions}, which   satisfies KL property and 
  covers most common mathematical programming objectives \cite{boltedanlew17,bolte2007clarke}.   
The definition of semialgebraic functions is provided below. 

 \begin{definition}[Semi-algebraic functions] A subset of $\mathbb{R}^n$ is called semi-algebraic if it can be written as 
 a finite union of sets of the form 
 \[\{ x\in \mathbb{R}^n : h_i(x) = 0, \  q_i(x)<0, \  i=1,\ldots, p\},\]
 where $h_i, q_i$ are real polynomial functions.  A function $f: \mathbb{R}^n \to \mathbb{R}\cup\{+\infty\}$ is semi-algebraic 
 if its graph is a semi-algebraic subset of $\mathbb{R}^{n+1}$.  
 \end{definition} 
  Semi-algebraic functions satisfy KL property with $\phi(s) = cs^{1-\theta}$, for some $\theta\in[0,1)\cap \mathbb{Q}$ and some $c> 0$ (see \citep{boltedanlew17,bolte2007clarke}),  and  finite sums of semi-algebraic 
 functions are semi-algebraic. 
  This nonsmooth result 
 generalizes the famous {\L}ojasiewicz inequality for real-analytic function \citep{lojasiewicz1963propriete}.

\section{Proximal Iteratively Reweighed $\ell_1$ Method with Extrapolation }\label{sec.proximal}

In this section, we propose   an \emph{e}xtrapolated  \emph{i}teratively \emph{r}eweighted 
$\ell_1$ algorithm, hereinafter named as EIRL1.  
The framework of this algorithm is presented in Algorithm \ref{alg.acc}.
\begin{algorithm}[htp]
	\caption{ Extrapolated Proximal Iteratively Reweighted $\ell_1$  Algorithm}
	\label{alg.acc}
	\begin{algorithmic}[1]
		\STATE{\textbf{Input:} $\mu\in(0,1)$, $\beta> L_f$, $\epsilon^0\in\mathbb{R}^n_{++}$, $0 \le \alpha^k\le \bar \alpha < 1$ and $x^0 $.}
		\STATE{\textbf{Initialize: set $k=0$, $x^{-1}=x^0$}. }
		\REPEAT
		\STATE Compute new iterate: 
		\begin{align}
	    w_i^k & = p(|x^k_i|+\epsilon^k_i)^{p-1},\\
		y^k & = x^k+\alpha^k(x^k-x^{k-1}),\label{eq:acc1}  \\ 
		x^{k+1} & \gets \underset{x\in \mathbb{R}^n}{\text{argmin}} \   \big\{   \nabla f(y^k)^T x  + \frac{\beta}{2}\|x-y^k\|^2 +\lambda \sum_{i=1}^{n}w_i^k|x_i|\big\},  \label{eq:acc2}
		\end{align}
		\STATE Choose $\epsilon^{k+1}\leq  \mu \epsilon^k$ and $0 \le \alpha^k\le \bar \alpha < 1$.
		\STATE Set $k\gets k+1$. 
		\UNTIL{convergence}
	\end{algorithmic} 
\end{algorithm}

Define the smooth approximation $F(x, \epsilon)$ of $F(x)$ with smoothing parameter $\epsilon$ as 
\[
F(x, \epsilon) := f(x)+\lambda \sum_{i=1}^n(|x_i|+\epsilon)^p
\]
and define the function of combining the  objective   with a proximal term as  
\[\begin{aligned}
\psi(x,y,\epsilon) := F(x,\epsilon) +\frac{\beta}{2}\|x-y\|_2^2,
\end{aligned}
\]
Before proceeding to the convergence analysis,  we first provide some  properties 
of our proposed method.  In the remainder of this paper, we make the following assumptions 
about $f$ and $F$. 
\begin{assumption}\label{ass.lip}
\begin{enumerate}
\item[(i)]  $f$ is Lipschitz differentiable with constant $L_f \ge 0$.
\item[(ii)] The initial point $(x^0, \epsilon^0)$ and $\beta$  are chosen  such that    
	$ \Lcal(F^0):= \{ x \mid F(x)  \le F^0\}$ is bounded where $F^0 := F(x^0, \epsilon^0)$ and 
	$\beta > L_f$.
\end{enumerate}
\end{assumption}


The following properties  hold true for Algorithm \ref{alg.acc}.

\begin{lemma}\label{lem.acc.gc}
	Suppose Assumption \ref{ass.lip}   holds true and 
	$\{x^k\}$ is generated by Algorithm \ref{alg.acc} for solving \eqref{prob.lp}. Then the following statements hold.
	\begin{enumerate}
		\item[(i)] $ \psi(x^k,x^{k-1},\epsilon^k)  -  \psi(x^{k+1},x^{k},\epsilon^{k+1})  \ge   \frac{1}{2}\beta(1-\bar \alpha^2)\|x^k-x^{k-1}\|^2$. 
		\item[(ii)] The sequence $\{x^k\} \subset \Lcal(F^0)$ and is bounded.
		\item[(iii)] $\lim\limits_{k\to \infty}\|x^{k+1}-x^k\|_2 = 0$.
		\item[(iv)] $\lim\limits_{k\to \infty}\|y^k-x^k\|_2 = 0$ and $\lim\limits_{k\to \infty}\|y^{k-1}-x^k\|_2 = 0$.
	\end{enumerate}
\end{lemma}

\begin{proof}
(i)   Since $x^{k+1}$ is the optimal solution of subproblem \eqref{eq:acc2}, there exists $\xi^{k+1} \in \partial |x^{k+1}| $ such that 
	\begin{equation}\label{eq:accgc1}
	0= \nabla f(y^k) +\beta(x^{k+1}-y^k)+ \lambda w^k\circ \xi^{k+1}, 
	\end{equation}
	which combined with the  strongly convexity of  \eqref{eq:acc2} yields  
	\begin{equation}\label{eq:accgc2}
	\begin{aligned}
	&\ \langle \nabla f(y^k), x^{k+1} \rangle + \frac{\beta}{2}\|x^{k+1}-y^k\|^2 + \lambda \sum_{i=1}^{n}w_i^k|x_i^{k+1}|\\
	\leq &\ \langle\nabla f(y^k), x^{k} \rangle + \frac{\beta}{2}\|x^{k}-y^k\|^2 + \lambda \sum_{i=1}^{n}w_i^k|x_i^k| 
	-\frac{\beta}{2}\|x^{k+1}-x^k\|^2.
	\end{aligned}
	\end{equation}
	From the concavity of $a^p$ on $\mathbb{R}_{++}$, 
	we know  for any $i\in \{1,\ldots, n\}$
	\[\begin{aligned}
	(|x_i^{k+1}|+\epsilon_i^k)^p \le &\ (|x_i^k |+\epsilon_i^k)^p + p (|x_i^k |+\epsilon_i^k)^{p-1} (|x_i^{k+1} | - |x_i^k|)\\
	= &\ (|x_i^k |+\epsilon_i^k)^p +  w_i^k (|x_i^{k+1} | - |x_i^k|).
	\end{aligned} \] 
	Summing the above inequality over all $i$ yields 
	\begin{equation}\label{eq:accgc3}
	\begin{aligned}
	\sum_{i=1}^n ( |x_i^{k+1}|+\epsilon_i^k)^p   &\le  \sum_{i=1}^n (|x_i^k |+\epsilon_i^k)^p +  \sum_{i=1}^n w_i^k (|x_i^{k+1} | - |x_i^k|).
	\end{aligned}
	\end{equation}
	Combining \eqref{eq:accgc2} with \eqref{eq:accgc3}, 
	\begin{equation}\label{eq:accgc4}
	\begin{aligned}
	&\ \langle \nabla f(y^k), x^{k+1} \rangle + \frac{\beta}{2}\|x^{k+1}-y^k\|^2 + \lambda \sum_{i=1}^n ( |x_i^{k+1}|+\epsilon_i^k)^p \\
	\leq 	&\  \langle\nabla f(y^k), x^{k} \rangle + \frac{\beta}{2}\|x^{k}-y^k\|^2 + \lambda \sum_{i=1}^{n}(|x_i^k |+\epsilon_i^k)^p
	-\frac{\beta}{2}\|x^{k+1}-x^k\|^2. 
	\end{aligned}
	\end{equation}
	It then follows that 
	\begin{equation*}\label{eq:accgc6}
	\begin{aligned}
	&\ F(x^{k+1},\epsilon^{k+1}) \\
	= &\ f(x^{k+1})+\lambda \sum_{i=1}^{n} (|x_i^{k+1}|+\epsilon_i^{k+1})^p\\
	\leq &\ f(y^k)+\langle \nabla f(y^k), x^{k+1}-y^k\rangle +\frac{L_f}{2}\|x^{k+1}-y^k\|^2+\lambda \sum_{i=1}^{n} (|x_i^{k+1}|+\epsilon_i^{k})^p\\
	\leq &\ f(y^k)+\langle \nabla f(y^k), x^{k+1}-y^k\rangle +\frac{\beta}{2}\|x^{k+1}-y^k\|^2+\lambda \sum_{i=1}^{n} (|x_i^{k+1}|+\epsilon_i^{k})^p\\
	\leq&\ f(y^k)+\langle \nabla f(y^k), x^{k}-y^k\rangle +\frac{\beta}{2}\|x^{k}-y^k\|^2+\lambda \sum_{i=1}^{n} (|x_i^{k}|+\epsilon_i^{k})^p-\frac{\beta}{2}\|x^{k+1}-x^k\|^2 \\
	\leq&\ f(x^k) + \lambda \sum_{i=1}^{n} (|x_i^{k}| +\epsilon_i^{k})^p +\frac{\beta}{2}\|x^{k}-y^k\|^2-\frac{\beta}{2}\|x^{k+1}-x^k\|^2\\
	=&\ F(x^k,\epsilon^k)+ \frac{\beta}{2}\|x^{k}-y^k\|^2-\frac{\beta}{2}\|x^{k+1}-x^k\|^2,
	\end{aligned}
	\end{equation*}
	where the first inequality follows from the Lipschitz differentiability of $f$, 
	the second inequality is by $\beta > L_f$, and 
	the  third  inequality follows from \eqref{eq:accgc4} and the last inequality is by the convexity of  $f$. 
	This means that 
	\[
	F(x^{k+1},\epsilon^{k+1})\leq F(x^{k},\epsilon^{k})+ \frac{\beta}{2}(\alpha^k)^2\|x^{k}-x^{k-1}\|^2-\frac{\beta}{2}\|x^{k+1}-x^k\|^2
	\]
	by the definition of $y^k$, which implies  that  
	\begin{equation}\label{eq:accgc5}
	\begin{aligned}
	   & \  \psi(x^k, x^{k-1}, \epsilon^k) - \psi(x^{k+1}, x^k, \epsilon^{k+1}) \\
	=& \ F(x^{k},\epsilon^{k})+\frac{\beta}{2}\|x^{k}-x^{k-1}\|^2 -\Big[F(x^{k+1},\epsilon^{k+1})+\frac{\beta}{2}\|x^{k+1}-x^{k}\|^2\Big]\\
	\geq&\ \frac{\beta}{2}(1-(\alpha^k)^2)\|x^k-x^{k-1}\|^2\\
	\ge & \  \frac{\beta}{2}(1-\bar \alpha^2)\|x^k-x^{k-1}\|^2
	\end{aligned}
	\end{equation} 
	by $\{\alpha^k\}\subset [0,\bar \alpha)$.  We then deduce from \eqref{eq:accgc5} and $0 < \bar \alpha < 1$ that the sequence $\{F(x^k,\epsilon^{k})+ \frac{\beta}{2}\|x^{k}-x^{k-1}\|^2 \}$ is monotonically decreasing. This proves part (i).

 (ii)   With $x^0=x^{-1}$, we know that for all $k\geq 0$,
	\[
	F(x^k) \le F(x^k,\epsilon^k)\leq F(x^k,\epsilon^k)+\frac{\beta}{2}\|x^k-x^{k-1}\|^2 \leq F(x^0,\epsilon^0)
	\]
	Under Assumption \ref{ass.lip}(ii), we know that   $\{x^k\} \subset \Lcal(F^0)$ and  is bounded.  This completes the proof of  part (ii).

  (iii) 	Summing both side of \eqref{eq:accgc5} from $0$ to $t$, we obtain that 
	\[
	\frac{\beta}{2}\sum_{k=0}^{t}(1-\bar \alpha^2)\|x^k-x^{k-1}\|^2  \le F(x^0,\epsilon^0)-F(x^t,\epsilon^t) \le F(x^0,\epsilon^0)- \underline 
	F,
	\]
	yielding   $\lim\limits_{k\to \infty}\|x^{k+1}-x^k\|=0$. Therefore, part (ii) is true.

(iv)  This part  is straightforward by noticing that 
	\[\begin{aligned}
	y^k - x^k & = \alpha^k(x^k-x^{k-1})\to 0\\  
	y^{k-1} - x^k & = (x^{k-1} - x^k) + \alpha^k (x^{k-1} - x^{k-2})\to 0.
	\end{aligned}\]
	from part (iii). 
\end{proof}

Using similar arguments from \cite{wang2019relating} and  Lemma \ref{lem.acc.gc}, we can also obtain results of  local stable support and sign as shown in  \cite{wang2019relating}, which are listed   below. 
 It shows that after some iteration $K$, $\{x^k\}_{k\ge K}$ stays in the same orthant, and the nonzero components 
 are bounded away from 0.   
\begin{theorem} \label{thm.also.stable}	 
	Suppose Assumption \ref{ass.lip}  is true and 
	$\{x^k\}$  is generated by Algorithm \ref{alg.acc}.  There then exists  $C>0$ and $K\in\mathbb{N}$ such that the 
	following statements hold true. 
	\begin{enumerate} 
		\item[(i)]  If $w_i^{\tilde k} > C/\lambda$, then  $x_i^k \equiv 0$ for all $k > \tilde k$. 
		\item[(ii)] There exists index sets $\Ical^*\cup\Acal^*=\{1,\ldots,n\}$ such that  
		$\Ical(x^k) \equiv \Ical^*$ and $\Acal(x^k) \equiv \Acal^*$ for any $k > K$. 
		\item[(iii)] For each $i \in \Ical^*$ and any $k > K$,
		\begin{equation}\label{eq.xboundeps}
		 \  |x_i^k| \ge \left(\frac{C}{p\lambda}\right)^{\frac{1}{p-1}}-\epsilon_i^k>0.
		\end{equation}
		\item[(iv)] For any limit point $x^*$ of $\{x^k\}$, it holds that $\Ical(x^*) = \Ical^*$, $\Acal(x^*) = \Acal^*$ and 
		\begin{equation}\label{eq.xbound}
		|x_i^*| \ge \left(\frac{C}{p\lambda}\right)^{\frac{1}{p-1}}, \  i\in\Ical^*.
		\end{equation}  
		 
		\item[(v)] There exists $s\in\{-1,0,+1\}^n$ such that $\text{sign}(x^k) \equiv s$ for any $k>K$. 
	\end{enumerate}
\end{theorem}
\begin{proof}
	
	The bounded $\{x^k\}$ (Lemma \ref{alg.acc} (ii)) implies that $\{y^k\}$ is also bounded. Then, there must exist $C>0$ such that for any $k$
	\begin{equation}\label{eq: bounded}
	\|\nabla f(y^k)+\beta (x^{k+1}-y^k)\|_{\infty}<C.
	\end{equation}
	
	If $w_i^{\tilde  k} > C/\lambda $ for some ${\tilde  k} \in \mathbb{N}$,  then  the optimality condition \eqref{eq:accgc1} 
	implies $x_i^{{\tilde  k} +1} = 0$. Otherwise we have 
	$ | \nabla f(y^{\tilde{k}})+\beta (x^{\tilde{k}+1}-y^{\tilde{k}}) | =   \lambda w_i^{\tilde  k}  > C $, 
	contradicting \eqref{eq: bounded}.  Monotonicity of $(\cdot)^{p-1}$ and  $0 + \epsilon_i^{\tilde k +1} \le |x_i^{\tilde k}| + \epsilon_i^{\tilde k}$ yield  
	\[w_i^{{\tilde  k}+1} =  p(0+ \epsilon_i^{{\tilde  k}+1})^{p-1}  \ge   p( |x_i^{\tilde k}| + \epsilon_i^{\tilde  k})^{p-1}  =   w_i^{\tilde  k} > C/\lambda.\]
	By induction  we know that  
	$x_i^k \equiv 0$ for any $k > \tilde  k$.  This completes the proof of (i).

	(ii)  Suppose by contradiction this statement is not true.  There exists $j \in \{1,\ldots, n\}$ such 
	that $\{x_j^k\}$ takes zero and nonzero values both for infinite times.  
	Hence, 
	there exists a subsequence $\Scal_1 \cup \Scal_2 = \mathbb{N}$ such that 
	$|\Scal_1|=\infty$, $|\Scal_2|=\infty$  and that  
	\[ x_j^k = 0, \forall k \in \Scal_1 \ \text{ and }\  x_j^k \neq 0, \forall k \in  \Scal_2.\]
	Since $\{ \epsilon_j^k \}_{\Scal_1}$ is monotonically decreasing to 0, there exists $\tilde k \in \Scal_1$ such that 
	\[ w_j^{\tilde k}  = p(|x_j^{\tilde k}| + \epsilon_j^{\tilde k} )^{p-1} = p (\epsilon_j^{\tilde k})^{p-1} > C/\lambda. \]
	It follows that $x_j^k \equiv 0$ for any $k > \tilde k$ by (i) which implies $\{ \tilde k + 1, \tilde k + 2, \ldots\} \subset \Scal_1$ and $|\Scal_2| < \infty$. This contradicts  the assumption    $|\Scal_2| = \infty.$ Hence, (ii) is true.

	(iii)  Combining  (i) and (ii), we know for any   $i\in \Ical^*$, $w_i^k \le  C/\lambda$, which is equivalent to \eqref{eq.xboundeps}. This proves (iii).

	(iv)  By (ii), $\Acal^*\subseteq \Acal(x^*)$ for any limit point $x^*$.  It follows from (ii) and (iii) that 
	$\Ical^* \subseteq \Ical(x^*)$ for any limit point $x^*$.  Hence, $\Acal^*= \Acal(x^*)$ and 
	 $\Ical^* = \Ical(x^*)$ for any limit point $x^*$ since $\Acal^*\cup \Ical^* = \{1,...,n\}$. 
	
	(v) By (iii) and Lemma \ref{alg.acc}(iii), there exists sufficiently large   $\bar k  $, such that  for any $k> \bar k$
	\begin{align}
		 & \  |x_i^k| >   \  \bar \epsilon: =  \frac{1}{2} \left(\frac{C}{p \lambda }\right)^{\frac{1}{p-1}} , \quad  \forall i\in \Ical^*. \label{coro.support} \\ 
\text{ and }	&\ \|x^{k+1} - x^k\|_2 <  \ \bar \epsilon\label{contradiction here}
	\end{align}
		We prove (v) by contradiction. Assume there exists   $j\in\Ical^*$ such that  the sign of $x_j$ 
	changes after $\bar k$.  
	Hence there must be  $\hat k  \ge   \bar k  $ such that $x_j^{\hat k}x_j^{\hat k+1} < 0 $. 
	It follows that  
	\[  \|x^{\hat k+1} - x^{\hat k}\|_2  \ge  |x_j^{\hat k+1} - x_j^{\hat k}| = \sqrt{ (x_j^{\hat k+1} )^2   + (x_j^{\hat k})^2  - 2x_j^{\hat k} x_j^{\hat k+1} }
	> \sqrt{ \bar\epsilon ^2   + \bar\epsilon ^2}
	=    \sqrt{2}\bar\epsilon,\]
	where the last inequality is by \eqref{coro.support}.    
	This contradicts  with \eqref{contradiction here}; hence 
	$\{x^k\}_{k\ge \bar k}$ have the same sign.  
	Without loss of generality, we can reselect $K=\bar k$ and then (v)  holds true. 
\end{proof}

%

\section{Global convergence}

Defining $\chi$ as the set of all cluster points of 
$\{x^k\}$, we now show 
the global convergence of Algorithm \ref{alg.acc}. 

\begin{theorem}[Global convergence]\label{thm.global.acc} Suppose Assumption \ref{ass.lip}   holds true and 
	$\{x^k\}$ is generated by Algorithm  \ref{alg.acc}. 
	The following statements hold true 
	\begin{enumerate}
		\item[(i)]  $F$ attains the same value at every cluster point of $\{x^k\}$, i.e., there exists  $\zeta \in \mathbb{R}$ 
		such that 
		$F(x^*,0)= \zeta$ for   any  $x^*\in \chi$. 
		\item[(ii)] 
		Each   point $x^*\in\chi$ is a stationary point of $F(x,0)$.
	\end{enumerate}
\end{theorem}

\begin{proof} (i)  For any $x^*\in \chi$ with subsequence $\{x^k\}_{\Scal}\to x^*$,    from Lemma \ref{lem.acc.gc} we know that 
	\begin{align*}
	& \lim_{k\to \infty \atop k\in\Scal} F(x^k,\epsilon^{k}) \\
	 =  &\lim_{k\to \infty \atop k\in\Scal} F(x^k,\epsilon^{k})+ \frac{\beta}{2}\|x^{k}-x^{k-1}\|^2 \\
	= &   \lim_{k\to \infty \atop k\in\Scal} \psi(x^k,  x^{k-1},\epsilon^k),
	\end{align*}
	by the monotonicity of $\psi(x^k,  x^{k-1},\epsilon^k)$ from Lemma \ref{lem.acc.gc} (i), we know there is a unique limit value $\zeta$, i.e. 
	$ \zeta = \lim\limits_{k\to \infty \atop k\in\Scal} F(x^k,\epsilon^{k}) $
	and  $F(x^*,0)= \zeta$ for any $x^*\in\chi$. 
	
	(ii) 
	Let $x^*$ be a limit point of $\{x^k\}$ with subsequence $\{x^k\}_{\Scal} \to x^*$. 
	We have for any $   i\in \Ical(x^*)$, 
	\[\begin{aligned}
	&\  \nabla_if(x^*) + \lambda p |x_i^*|^{p-1}\sign(x_i^*)\\
	= &     \lim_{k\to\infty \atop k\in\Scal}	 \nabla_if(x^k) + \lambda p |x_i^k|^{p-1}\sign(x^k_i) \\
	= &     \lim_{k\to\infty \atop k\in\Scal}	 \nabla_if(y^k)  + \lambda p(|x_i^k|+ \epsilon_i^k)^{p-1} \text{sign}(x_i^{k+1}) \\
	= & \lim_{k\to\infty \atop k\in\Scal} -\beta(x_i^{k+1}-y_i^k) \\
	= &\ 0,
	\end{aligned} \]
	the first and second equality is by Lemma  \ref{lem.acc.gc}(iii)-(iv), the 
	third equality is due to 
	%
	$x^{k+1}$  satisfying the optimal condition of the subproblem  for $k> K$
	\begin{equation}\label{eq.831.1}
	\nabla_if(y^{k})+ \beta (x^{k+1} - y^k) + \lambda p(|x_i^k|+ \epsilon_i^k)^{p-1}\text{sign}(x_i^{k+1})  = 0,  \quad    i\in \Ical(x^*) \\
	\end{equation}
	and last equality is  by Lemma \ref{lem.acc.gc}(iv). 
	Therefore,  $x^*$ is first-order optimal, completing the proof. 
\end{proof}

 To further analyze the property of $\{(x^k,\epsilon^k)\}$,  denote $\delta_i = \sqrt{\epsilon_i}$ since $\epsilon_i$ is restricted to be nonnegative and write $F $ and $\psi$ as   functions  of $(x, \delta)$  for simplicity, i.e.,     
\[\begin{aligned}
F(x,\delta) & =  f(x)+\lambda \sum_{i=1}^n(|x_i|+\delta_i^2)^p, \\
\psi(x,y,\delta) & = f(x)+\lambda \sum_{i=1}^n(|x_i|+\delta_i^2)^p+\frac{\beta}{2}\|x-y\|_2^2.
\end{aligned}
\]

 Next, we show the uniqueness of the limit points of $\{x^k\}$ under KL property. 
Notice that after   the $K$th iteration, the iterates $\{x^k_{\Ical^*}\}$ remains in the interior of the 
same orthant of $\mathbb{R}^{|\Ical^*|}$ and are bounded away from
 the axes by   Theorem \ref{thm.also.stable}.
  Hence we can  assume 
 the reduced function $F(x_{\Ical^*}, \delta_{\Ical^*})$ has the KL property at $(x_{\Ical^*}^*, 0_{\Ical^*}) \in \mathbb{R}^{2|\Ical^*|} $, which is a weaker 
 condition than assuming the KL property of $F(x,0)$ at $(x^*, 0) \in \mathbb{R}^{2n}$.

  \begin{assumption}\label{ass.KL}   
 Suppose $\psi$ has the KL property at every 
\[ (x_{\Ical^*}^*, x_{\Ical^*}^*, 0_{\Ical^*}) \in \mathbb{R}^{3|\Ical^*|}, \quad \forall x^*\in\chi .\] 
\end{assumption}
%

%
%
%
%
%
%
%
%
%
%
%
%


By Theorem \ref{thm.also.stable},  for all sufficiently large $k$,  the components of $\{x^k_{\Ical^*}\} $ are all uniformly bounded away from 0 and 
$x^k_{\Acal^*}\equiv 0$. 
We have the following properties about $\psi$. 
\begin{lemma}\label{lem.acc.gc2}
	Let $\{x^k\}$ be a sequence generated by Algorithm  \ref{alg.acc} and Assumption \ref{ass.KL} is satisfied. 
	For sufficiently large $k$, the following statements hold.
	\begin{enumerate}
		\item[(i)]  There exists $D_1>0$ such that for all $k$
		\begin{align*}\label{eq:acc.gc4}
		& \|\nabla \psi(x^k_{\Ical^*}, x^{k-1}_{\Ical^*},\delta^k_{\Ical^*})\|_2\\
		 \leq & D_1(\|x^k_{\Ical^*}-x^{k-1}_{\Ical^*}\|_2 +\|x^{k-1}_{\Ical^*}-x^{k-2}_{\Ical^*}\|_2+\|\delta^{k-1}_{\Ical^*}\|_1-\|\delta^{k}_{\Ical^*}\|_1),
		\end{align*}
		so that $\lim\limits_{k\to \infty} \|\nabla \psi(x^k_{\Ical^*}, x^{k-1}_{\Ical^*},\delta^k_{\Ical^*})\|=0$. 
		\item[(ii)]   $\{\psi (x^k_{\Ical^*},x^{k-1}_{\Ical^*}, \delta^k_{\Ical^*})\}$ is monotonically decreasing. There exists   $D_2>0$ such that 
		\begin{equation*}\label{eq:acc.gc5}
		\psi(x^k_{\Ical^*},x^{k-1}_{\Ical^*},\delta^k_{\Ical^*})-\psi(x^{k+1}_{\Ical^*},x^k_{\Ical^*}, \delta^{k+1}_{\Ical^*}) \geq D_2\|x^k_{\Ical^*}-x^{k-1}_{\Ical^*}\|^2_2.
		\end{equation*}
		\item[(iii)]  $\psi(x^*_{\Ical^*}, x^*_{\Ical^*},0_{\Ical^*})= \zeta: = \lim\limits_{k\to\infty}\psi(x^k_{\Ical^*}, x^{k-1}_{\Ical^*},\delta^{k}_{\Ical^*})$, where $\Gamma$ is the set of 
		the cluster points of $\{(x^k_{\Ical^*}, x^{k-1}_{\Ical^*}, \delta^{k}_{\Ical^*})\}$, i.e., $\Gamma:=\{(x^*_{\Ical^*},x^*_{\Ical^*},0_{\Ical^*})\mid x^*\in\chi\}$.
	\end{enumerate}
\end{lemma}

\begin{proof}
	
  (i)  Notice for sufficiently large $k$,  the gradient of  $\psi$ at $(x^k_{\Ical^*},x^{k-1}_{\Ical^*}, \delta^k_{\Ical^*})$ is 
	\begin{equation}
	\label{grad.xydelta}
	\begin{aligned}
	\nabla_x \psi(x^k_{\Ical^*},x^{k-1}_{\Ical^*}, \delta_{\Ical^*}^k) & =\nabla f(x_{\Ical^*}^k)+\beta(x_{\Ical^*}^k -x_{\Ical^*}^{k-1})+\lambda w_{\Ical^*}^k\circ  \text{sign}(x_{\Ical^*}^{k}), \\
	\nabla_y \psi(x_{\Ical^*}^k,x_{\Ical^*}^{k-1}, \delta_{\Ical^*}^k) & =   -\beta(x_{\Ical^*}^k-x_{\Ical^*}^{k-1}),\\
	\nabla_\delta \psi(x_{\Ical^*}^k,x_{\Ical^*}^{k-1}, \delta_{\Ical^*}^k) & = 2\lambda w_{\Ical^*}^k \circ \delta_{\Ical^*}^k.
	\end{aligned}
	\end{equation}
	We first derive the upper bound for $\|\nabla_x\psi(x_{\Ical^*}^k,x_{\Ical^*}^{k-1},\delta_{\Ical^*}^k)\|_2$. The first-order optimality condition of the $(k-1)$th subproblem at $x_{\Ical^*}^k$ is 
	\begin{equation*}
	\label{eq:acc.gc2}
	\nabla f(y_{\Ical^*}^{k-1})+\beta(x_{\Ical^*}^{k}-y_{\Ical^*}^{k-1})+\lambda w_{\Ical^*}^{k-1}\circ \text{sign}(x_{\Ical^*}^{k}) = 0. 
	\end{equation*}
	Hence, we have
	\begin{equation}
	\label{eq:acc.gc3}
	\begin{aligned}
	 \nabla_x\psi(x_{\Ical^*}^k,x_{\Ical^*}^{k-1},\delta_{\Ical^*}^k)  
	= &\ \nabla f(x_{\Ical^*}^k)-\nabla f(y_{\Ical^*}^{k-1}) \\
	& + \beta(y_{\Ical^*}^{k-1}-x_{\Ical^*}^{k-1})+\lambda (w_{\Ical^*}^k-w_{\Ical^*}^{k-1})\circ \text{sign}(x_{\Ical^*}^{k}).
	\end{aligned}
	\end{equation}
	By Lemma \ref{lem.acc.gc}(iv) and the Lipschitz differentiability of $f$, we know 
	\begin{equation}\label{eq.nablaf.1}
	\begin{aligned}
	\| \nabla f(x_{\Ical^*}^k)-\nabla f(y_{\Ical^*}^{k-1})\|_2 & \le L_f \|x_{\Ical^*}^k - y_{\Ical^*}^{k-1}\|_2\\
	& = L_f  \|x_{\Ical^*}^k - x_{\Ical^*}^{k-1} - \alpha^k(x_{\Ical^*}^{k-1}-x_{\Ical^*}^{k-2})\|_2 \\
	& = L_f \|x_{\Ical^*}^k-x_{\Ical^*}^{k-1}\|_2 
	+  L_f \bar \alpha\|x_{\Ical^*}^{k-1}-x_{\Ical^*}^{k-2}\|_2,
	\end{aligned}\end{equation}
	and that 
	\begin{equation}\label{eq.nablaf.2}
	\|   \beta(y_{\Ical^*}^{k-1}-x_{\Ical^*}^{k-1}) \|_2 = \beta \|x_{\Ical^*}^{k-1} + \alpha^k(x_{\Ical^*}^{k-1}-x_{\Ical^*}^{k-2})- x_{\Ical^*}^{k-1}\|_2 \le \beta \bar \alpha \|x_{\Ical^*}^{k-1}-x_{\Ical^*}^{k-2}\|_2.
	\end{equation}
	On the other hand, by Lagrange's mean value theorem,  for each $i \in {\Ical^*}$, 
	\begin{equation*}
	\begin{aligned}
	\big| (w_i^k-w_i^{k-1})\cdot \text{sign}(x_i^k) \big|&= \big| w_i^k-w_i^{k-1}\big| \\
	& =   \big|p(|x_i^k|+(\delta_i^{k})^2)^{p-1}-p(|x_i^{k-1}|+(\delta_i^{k-1})^2)^{p-1}\big|\\
	& =   \big|p(1-p)(c_i^k)^{p-2} \Big[|x_i^k|-|x_i^{k-1}|+(\delta_i^{k})^2-(\delta_i^{k-1})^2\Big]\big|\\
	&\leq   p(1-p)(c_i^k)^{p-2} \Big[|x_i^k-x_i^{k-1}|+ (\delta_i^{k-1})^2-(\delta_i^{k})^2 \Big]\\
	&\leq   p(1-p)(c_i^k)^{p-2} \Big[|x_i^k-x_i^{k-1}|+ 2\delta_i^0(\delta_i^{k-1}-\delta_i^{k} )\Big],
	\end{aligned}
	\end{equation*}	
	where the first equality is by the fact that $x_i^k\neq 0$, and 
	$c_i^k$ is between $|x_i^k|+(\delta_i^{k})^2$ and $|x_i^{k-1}|+(\delta_i^{k-1})^2$. 
	It then follows that 
	\begin{equation}\label{eq.nablaf.3}
	\begin{aligned}
	& \  \| (w_{\Ical^*}^k-w_{\Ical^*}^{k-1})\circ \text{sign}(x_{\Ical^*}^k) \|_2\\
	\leq  &\ 	\| (w_{\Ical^*}^k-w_{\Ical^*}^{k-1})\circ \text{sign}(x_{\Ical^*}^k) \|_1 \\
	\leq   &\  \bar C  \left(\|x_{\Ical^*}^k-x_{\Ical^*}^{k-1}\|_1 + 2\|\delta_{\Ical^*}^0\|_\infty( \|\delta_{\Ical^*}^{k-1}\|_1 - \|\delta_{\Ical^*}^{k}\|_1)\right)\\
	\leq   &\   \bar C \left(\sqrt{n} \|x_{\Ical^*}^k-x_{\Ical^*}^{k-1}\|_2 + 2\|\delta_{\Ical^*}^0\|_\infty( \|\delta_{\Ical^*}^{k-1}\|_1 - \|\delta_{\Ical^*}^{k}\|_1)\right),	\end{aligned}
	\end{equation}
	where the second inequality 
	is by $(c_i^k)^{p-2}  \le  \left( \frac{p\lambda}{C}\right)^{\frac{p-2}{1-p}}$  from Theorem \ref{thm.also.stable}(ii) with  $\bar C : = p(1-p)\left( \tfrac{p\lambda}{C}\right)^{\frac{p-2}{1-p}}$. 
	Combining \eqref{eq:acc.gc3}, \eqref{eq.nablaf.1},  \eqref{eq.nablaf.2} and \eqref{eq.nablaf.3}, we know 
	\begin{equation}
	\label{first.bound}
	\begin{aligned}
	\| \nabla_x\psi(x^k_{\Ical^*},x^{k-1}_{\Ical^*},\delta^k_{\Ical^*})\|_2 
	\le &\  (L_f +  \bar C\sqrt{n} )\|x_{\Ical^*}^k  -x^{k-1}_{\Ical^*}\|_2  \\
	& \  + \bar\alpha (L +\beta) \|x_{\Ical^*}^{k-1}-x_{\Ical^*}^{k-2}\|_2 \\
	&\ +  2 \bar C \|\delta^0\|_\infty (\|\delta_{\Ical^*}^{k-1}\|_1 - \|\delta^{k}_{\Ical^*}\|_1).
	\end{aligned}
	\end{equation}
	On the other hand, we have from  \eqref{grad.xydelta} that 
	\begin{equation}\label{second.bound}
	\| \nabla_y \psi(x^k_{\Ical^*}, x^{k-1}_{\Ical^*}, \delta^k_{\Ical^*}) \|_2 = \beta \| x^k_{\Ical^*} - x^{k-1}_{\Ical^*}\|_2,
	\end{equation} 
	and that 
	\begin{equation}\label{third.bound}
	\begin{aligned}
	\| \nabla_\delta \psi(x^k_{\Ical^*}, x^{k-1}_{\Ical^*}, \delta^k_{\Ical^*}) \|_2 & \le \| \nabla_\delta \psi(x_{\Ical^*}^k, x_{\Ical^*}^{k-1}, \delta_{\Ical^*}^k) \|_1\\
	& = \sum_{i\in \Ical^*} 2\lambda w_i^k \delta_i^k\\
	& \le \sum_{i\in\Ical^*} 2\lambda \tfrac{C}{\lambda} \frac{\sqrt{\mu}}{1-\sqrt{\mu}} (\delta_i^{k-1}-\delta_i^k)\\
	& \leq  \frac{2 C\sqrt{\mu}}{1-\sqrt{\mu}}(\|\delta_{\Ical^*}^{k-1}\|_1-\|\delta^k_{\Ical^*}\|_1), 
	\end{aligned}	
	\end{equation}	
	where the second inequality is by Theorem \ref{thm.also.stable}(ii) and $\delta^k \leq \sqrt{\mu} \delta^{k-1}$. 
	Overall,  we obtain from \eqref{first.bound}, \eqref{second.bound} and \eqref{third.bound} 
	that Part (i) holds true by setting 
	\[ D_1 = \max \left( \beta + L_f +  \bar C \sqrt{n},  
	\bar\alpha (L +\beta), 
	2\bar C \|\delta_{\Ical^*}^0\|_\infty+  \tfrac{2 C\sqrt{\mu}}{1-\sqrt{\mu}} \right).
	\]
	
	%

	%
	%
	
	%
	%
	
	Part (ii) and (iii) follow directly from Lemma \ref{lem.acc.gc}(i) with $D_2 = \beta(1-\bar\alpha^2)/2$
	and Theorem \ref{thm.global.acc}(i), respectively. 
\end{proof}


Now we are ready to prove the uniqueness of limit points.

\begin{theorem}\label{lem.acc.gc2.conv}
	Let $\{x^k\}$ be a sequence generated by Algorithm \ref{alg.acc} and Assumption \ref{ass.KL} is satisfied. 
	Then    $\{x^k\}$ converges to a stationary point of $F(x,0)$; moreover, $$\sum_{k=0}^{\infty}\|x^k-x^{k-1}\|_2<\infty.$$
\end{theorem}

\begin{proof}
	By Theorem \ref{thm.global.acc}, it suffices to show that 
	$\{x^k\}$ has a unique cluster point.  

	By Lemma \ref{lem.acc.gc2},   $\psi(x_{\Ical^*}^k,x_{\Ical^*}^{k-1},\delta_{\Ical^*}^{k})$ is monotonically decreasing and converging to $\zeta$.  
	If $\psi(x_{\Ical^*}^k,x_{\Ical^*}^{k-1},\delta_{\Ical^*}^{k})= \zeta$ after some $k_0$, then from  Lemma \ref{lem.acc.gc2}(ii), 
	we know $x_{\Ical^*}^{k+1} = x_{\Ical^*}^k$ for all $k > k_0$, meaning $x^k \equiv x^{k_0} \in \chi$, so that the proof is done.

	We next consider the case that $ \psi(x_{\Ical^*}^k,x_{\Ical^*}^{k-1},\delta_{\Ical^*}^{k})>\zeta$ for all $k$. Since $\psi$ has the KL property at every 
	$(x_{\Ical^*}^*,x_{\Ical^*},0_{\Ical^*})\in \mathbb{R}^{3|{\Ical^*}|}$, there exists a continuous concave function $\phi$ with $\eta >0$ and 
	neighborhood 
	\[ U=\{(x_{\Ical^*},y_{\Ical^*},\delta_{\Ical^*})\in \mathbb{R}^{3|{\Ical^*}|} \mid  \text{dist}((x_{\Ical^*},y_{\Ical^*},\delta_{\Ical^*}),\Gamma) <\tau \}\]      such that 
	\begin{equation}\label{eq:eiskl}
	\phi'(\psi(x_{\Ical^*},y_{\Ical^*},\delta_{\Ical^*})-\zeta)\text{dist}((0,0,0),\partial \psi(x_{\Ical^*},y_{\Ical^*},\delta_{\Ical^*}))\geq 1
	\end{equation} 
	for all $(x_{\Ical^*},y_{\Ical^*},\delta_{\Ical^*})\in U \cap\{(x_{\Ical^*},y_{\Ical^*},\delta_{\Ical^*})\in \mathbb{R}^{3|{\Ical^*}|}| \zeta< \psi(x_{\Ical^*},y_{\Ical^*},\delta_{\Ical^*}) < \zeta +\eta \}$.

	By the fact that $\Gamma$ is the set of limit points of $\{(x_{\Ical^*}^k,x_{\Ical^*}^{k-1},\delta_{\Ical^*}^{k})\}$ and Lemma \ref{lem.acc.gc}(ii), we have
	\[
	\lim_{k\to \infty} \text{dist}((x_{\Ical^*}^k,x_{\Ical^*}^{k-1},\delta_{\Ical^*}^{k}),\Gamma) =0.
	\]
	Hence, there exist $k_1\in\mathbb{N}$ such that $\text{dist}((x_{\Ical^*}^k,x_{\Ical^*}^{k-1},\delta_{\Ical^*}^{k}),\Gamma)<\tau$ for any $k>k_1$. On the other 
	hand, since  $\{\psi(x_{\Ical^*}^k,x_{\Ical^*}^{k-1},\delta_{\Ical^*}^{k})\}$ is monotonically decreasing  and converges to $\zeta$, there exists $k_2\in\mathbb{N}$ such that $\zeta < \psi(x_{\Ical^*}^{k},x_{\Ical^*}^{k-1},\delta_{\Ical^*}^{k}) <\zeta +\eta$ for all $k>k_2$. Letting $\bar{k}=\max \{k_1,k_2 \}$ and noticing that 
	$\psi$ is smooth at $(x_{\Ical^*}^k, x_{\Ical^*}^{k-1}, \delta_{\Ical^*}^k)$ for all $k>\bar k$,  we know from \eqref{eq:eiskl} that 
	\[
	\phi'(\psi(x_{\Ical^*}^k,x_{\Ical^*}^{k-1},\delta_{\Ical^*}^{k})-\zeta) \|\nabla \psi(x_{\Ical^*}^k,x_{\Ical^*}^{k-1},\delta_{\Ical^*}^{k})\|_2 \geq 1, \qquad \text{for all } k\geq \bar{k}. 
	\]
	It follows that for any $k\geq \bar{k}$, 
	\[
	\begin{aligned}
	& \Big[\phi( \psi(x_{\Ical^*}^k,x_{\Ical^*}^{k-1},\delta_{\Ical^*}^{k})-\zeta)-\phi(  \psi(x_{\Ical^*}^{k+1},x_{\Ical^*}^{k},\delta_{\Ical^*}^{k+1})-\zeta)\Big]\\
	&\ \cdot D_1\Big[\|x_{\Ical^*}^k-x_{\Ical^*}^{k-1}\|_2 + \|x_{\Ical^*}^{k-1}-x_{\Ical^*}^{k-2}\|_2 + \|\delta_{\Ical^*}^{k-1}\|_1 - \|\delta_{\Ical^*}^k\|_1 \Big]\\
	\geq &\ \Big[\phi( \psi(x_{\Ical^*}^k,x_{\Ical^*}^{k-1},\delta_{\Ical^*}^{k})-\zeta)-\phi(  \psi(x_{\Ical^*}^{k+1},x_{\Ical^*}^{k},\delta_{\Ical^*}^{k+1})-\zeta)\Big]\|\nabla \psi(x_{\Ical^*}^k,x_{\Ical^*}^{k-1},\delta_{\Ical^*}^{k})\|_2\\
	\geq \ &   \phi'(\psi(x_{\Ical^*}^k,x_{\Ical^*}^{k-1},\delta_{\Ical^*}^{k})-\zeta)\cdot 
	\|\nabla \psi(x_{\Ical^*}^k,x_{\Ical^*}^{k-1},\delta_{\Ical^*}^{k})\| \\&\cdot \Big[\psi(x_{\Ical^*}^k,x_{\Ical^*}^{k-1},\delta_{\Ical^*}^{k})-\psi(x_{\Ical^*}^{k+1},x_{\Ical^*}^k,\delta_{\Ical^*}^{k+1})\Big] \\
	\geq\ & \psi(x_{\Ical^*}^k,x_{\Ical^*}^{k-1},\delta_{\Ical^*}^{k})-\psi(x_{\Ical^*}^{k+1},x_{\Ical^*}^k,\delta_{\Ical^*}^{k+1})\\
	\geq\ & D_2 \|x_{\Ical^*}^k-x_{\Ical^*}^{k-1}\|_2^2,
	\end{aligned}
	\]
	where the first inequality is by Lemma \ref{lem.acc.gc2}(i), the second inequality is by the concavity of $\phi$, 
	, the fourth inequality is by Lemma \ref{lem.acc.gc2}(ii). 
	Rearranging and taking the square root of both sides,  and using the inequality of arithmetic 
	and geometric means, we have
	\begin{equation*}
	\begin{aligned}
	\|x_{\Ical^*}^k-x_{\Ical^*}^{k-1}\|_2 \leq &\ \sqrt{\frac{2D_1}{D_2}[\phi(\psi(x_{\Ical^*}^k,x_{\Ical^*}^{k-1},\delta_{\Ical^*}^{k})-\zeta)-\phi( \psi(x_{\Ical^*}^{k+1},x_{\Ical^*}^{k},\delta_{\Ical^*}^{k+1})-\zeta)]}\\ &\times \sqrt{\frac{\|x_{\Ical^*}^k-x_{\Ical^*}^{k-1}\|+\|x_{\Ical^*}^{k-1}-x_{\Ical^*}^{k-2}\|+(\delta_{\Ical^*}^{k-1}-\delta_{\Ical^*}^{k})}{2}}\\
	\leq &\ \frac{D_1}{D_2} \Big[\phi(\psi(x_{\Ical^*}^k,x_{\Ical^*}^{k-1},\delta_{\Ical^*}^{k})-\zeta)-\phi( \psi(x_{\Ical^*}^{k+1},x_{\Ical^*}^{k},\delta_{\Ical^*}^{k+1})-\zeta)\Big]\\ &
	\ + \frac{1}{4}[\|x_{\Ical^*}^k-x_{\Ical^*}^{k-1}\|_2
	+ \|x_{\Ical^*}^{k-1}-x_{\Ical^*}^{k-2}\|_2+ (\|\delta_{\Ical^*}^{k-1}\|_1-\|\delta_{\Ical^*}^{k}\|_1)]. 
	\end{aligned}
	\end{equation*}
	Subtracting $\frac{1}{2}\|x_{\Ical^*}^k-x_{\Ical^*}^{k-1}\|_2$ from both sides, we have 
	\begin{equation*} 
	\begin{aligned}
	\frac{1}{2}\|x_{\Ical^*}^k-x_{\Ical^*}^{k-1}\|_2  \leq &\ \frac{D_1}{D_2}[\phi(\psi(x_{\Ical^*}^k,x_{\Ical^*}^{k-1},\delta_{\Ical^*}^{k})-\zeta)-\phi( \psi(x_{\Ical^*}^{k+1},x_{\Ical^*}^{k},\delta_{\Ical^*}^{k+1})-\zeta)] \\
	&+ \frac{1}{4}(\|x_{\Ical^*}^{k-1}-x_{\Ical^*}^{k-2}\|_2-\|x_{\Ical^*}^k-x_{\Ical^*}^{k-1}\|_2)+\frac{1}{4}(\|\delta_{\Ical^*}^{k-1}\|_1-\|\delta_{\Ical^*}^{k}\|_1).
	\end{aligned}
	\end{equation*}
	Summing up both sides from $\bar k$ to $t$, we have 
	\begin{equation*}\label{eq:acc.gc6}
	\begin{aligned}
	\tfrac{1}{2}\sum_{k=\bar k}^t \|x_{\Ical^*}^k-x_{\Ical^*}^{k-1}\|_2  \leq &\ \frac{D_1}{D_2}[\phi(\psi(x_{\Ical^*}^{\bar k},x_{\Ical^*}^{\bar k-1},\delta_{\Ical^*}^{\bar k})-\zeta)-\phi( \psi(x_{\Ical^*}^{t+1},x_{\Ical^*}^{t},\delta_{\Ical^*}^{t+1})-\zeta)] \\
	&+ \tfrac{1}{4}(\|x_{\Ical^*}^{\bar k-1}-x_{\Ical^*}^{\bar k-2}\|_2-\|x_{\Ical^*}^t-x_{\Ical^*}^{t-1}\|_2)+\tfrac{1}{4}(\|\delta_{\Ical^*}^{\bar k-1}\|_1-\|\delta_{\Ical^*}^{t}\|_1).
	\end{aligned}
	\end{equation*}
	Now letting $t\to\infty$, we know $\|\delta_{\Ical^*}^{t}\|_1\to 0$ and $\|x_{\Ical^*}^t-x_{\Ical^*}^{t-1}\|_2\to 0$ by Lemma \ref{lem.acc.gc}(iii), 
	and that $\phi( \psi(x_{\Ical^*}^{t+1},x_{\Ical^*}^{t},\delta_{\Ical^*}^{t+1})-\zeta)\to \phi(\zeta-\zeta) = \phi(0) = 0$.
	Therefore, we know 
	\begin{equation}\label{sum.x.bound}
	\begin{aligned}
	\sum_{k=\bar k}^\infty \|x_{\Ical^*}^k-x_{\Ical^*}^{k-1}\|_2 &\le \tfrac{2D_1}{D_2}\phi(\psi(x_{\Ical^*}^{\bar k},x_{\Ical^*}^{\bar k-1},\delta_{\Ical^*}^{\bar k})-\zeta) 
	+  \tfrac{1}{2}(\|x_{\Ical^*}^{\bar k-1}-x_{\Ical^*}^{\bar k-2}\|_2 + \|\delta_{\Ical^*}^{\bar k-1}\|_1)
	\end{aligned}
	\end{equation}
	Since $I^*$ is the   support of $\{x^k\}_{k\ge \bar k}$ by Theorem \ref{thm.also.stable},
	$$\sum_{k=\bar{k}}^{\infty}\|x^k-x^{k-1}\|_2 = \sum_{k=\bar k}^\infty \|x_{\Ical^*}^k-x_{\Ical^*}^{k-1}\|_2 <\infty. $$
	This implies that $\{x^k\}$ is a Cauchy sequence and consequently 
	  is convergent. 
\end{proof}

%

\section{Local convergence rate}
Now we investigate the local convergence rate of  Algorithm \ref{alg.acc}  by assuming 
that $\psi$ has the property 
at $(x^*_{\Ical^*}, x^*_{\Ical^*}, 0_{\Ical^*})$ 
 with $\phi$ in the KL definition taking the form $\phi(s)=cs^{1-\theta}$ for some $\theta \in [0,1)$ and $c> 0$.
By the discussion in \S\ref{sec.KL}, 
this additional requirement is satisfied by the semi-algebraic functions, which 
is also commonly satisfied by a wide range of functions. 
For example, for any  $p\in\mathbb{Q}$, 
 $\sum_{i\in\Ical^*} (|x_i|+\epsilon_i)^p$ is semi-algebraic  around $(x^*_{\Ical^*}, 0_{\Ical^*})$ by \citep{wakabayashi2008remarks}.   
  Therefore,  $\psi(x_{\Ical^*}, y_{\Ical^*}, \delta_{\Ical^*})$ is semi-algebraic around $(x^*_{\Ical^*}, x^*_{\Ical^*}, 0_{\Ical^*})$ 
  if  $f(x_{\Ical^*})$ is semi-algebraic in a neighborhood around $x^*_{\Ical^*}$.

We now show that Algorithm \ref{alg.acc} has local linear convergence. 
\begin{theorem}
	Suppose $\{x^k\}$ is   generated by Algorithm \ref{alg.acc}  and converges to $x^*$. Assume that $\psi(x_{\Ical^*}, y_{\Ical^*}, \delta_{\Ical^*})$ 
	has the KL property at $(x^*_{\Ical^*}, x^*_{\Ical^*}, 0_{\Ical^*})$  with $\phi$ in the KL definition taking the form $\phi(s)=cs^{1-\theta}$ for some $\theta \in [0,1)$ and $c> 0$.  Then the following statements hold.
	\begin{enumerate}
		\item[(i)] If $\theta =0$, then there exists $k_0\in\mathbb{N}$ so that $x^k\equiv x^*$   for any $k>k_0$;
		\item[(ii)] If $\theta \in (0,\frac{1}{2}]$, then there exist $\gamma \in (0,1), c_1>0, c_2>0$ such that 
		\begin{equation}\label{linear.1}
		\|x^k-x^*\|_2< c_1\gamma^{k-1}- c_2\|\delta^{k}\|_1
		\end{equation} for sufficiently large $k$;
		\item[(iii)] If $\theta \in (\frac{1}{2},1)$, then there exist $c_3>0, c_4>0$ such that 
		\begin{equation}\label{linear.2}
		\|x^k-x^*\|_2< c_3 k^{-\frac{1-\theta}{2\theta-1}}-c_4\|\delta^k\|_1 
		\end{equation}   for sufficiently large $k$.
	\end{enumerate}
\end{theorem}
\begin{proof} 
	(i)  If $\theta=0$, then $\phi(s)=cs$ and $\phi'(s)\equiv c$.  We claim that there must exist $k_0>0$ such that $\psi(x^{k_0},x^{k_0-1},\delta^{k_0})=\zeta$. Suppose by contradiction this is 
	not true so that $\psi(x_{\Ical^*}^k,x_{\Ical^*}^{k-1},\delta_{\Ical^*}^{k})>\zeta$ for all $k$. Since $\lim\limits_{k\to \infty } x^k=x^*$ and the sequence $\{\psi(x_{\Ical^*}^k,x_{\Ical^*}^{k-1},\delta_{\Ical^*}^{k})\}$ is monotonically decreasing to $\zeta$ by Lemma \ref{lem.acc.gc2}.  
	The KL inequality implies that  all sufficiently large $k$, 
	\[
	c \|\nabla \psi(x_{\Ical^*}^k,x_{\Ical^*}^{k-1},\delta_{\Ical^*}^{k})\|_2 \geq 1,
	\]
	contradicting  $\|\nabla \psi(x_{\Ical^*}^k,x_{\Ical^*}^{k-1},\delta_{\Ical^*}^{k})\|_2 \to 0$ by Lemma \ref{lem.acc.gc2}(i).  
	Thus, there exists $k_0\in\mathbb{N}$ such that $\psi(x_{\Ical^*}^{k},x_{\Ical^*}^{k-1},\delta_{\Ical^*}^k)= \psi(x_{\Ical^*}^{k_0},x_{\Ical^*}^{k_0-1},\delta_{\Ical^*}^{k_0})=\zeta$ for all $k>k_0$. Hence, we conclude from  Lemma \ref{lem.acc.gc2}(ii) that $x_{\Ical^*}^{k+1} = x_{\Ical^*}^k$ for all $k > k_0$, meaning $x^k \equiv x^* = x^{k_0}$   for all $k \ge k_0$. This proves (i).

	(ii)-(iii) Now consider $\theta\in (0,1)$. First of all, if there exists $k_0\in\mathbb{N}$ such that $\psi(x_{\Ical^*}^{k_0},x_{\Ical^*}^{k_0-1}, \delta_{\Ical^*}^{k_0})=\zeta$, then using the same argument of the proof for (i),  we can see that $\{x^k\}$ converges finitely.   
	Thus, we only need to consider the case that $\psi(x_{\Ical^*}^k,x_{\Ical^*}^{k-1}, \delta_{\Ical^*}^k)>\zeta$ for all $k$.
	
	Moreover, define $S_k=\sum_{l=k}^{\infty}\|x_{\Ical^*}^{l+1}-x_{\Ical^*}^l\|_2$. It holds that 
	\[ \begin{aligned}
	\|x_{\Ical^*}^k - x_{\Ical^*}^*\|_2 = \|x_{\Ical^*}^k - \lim_{t\to\infty} x_{\Ical^*}^t\|_2 =  &  \|\lim_{t\to\infty}\sum_{l=k}^t( x_{\Ical^*}^{l+1} - x_{\Ical^*}^l)\|_2\\
	 \le & S_k := \sum_{l=k}^{\infty}\|x_{\Ical^*}^{l+1}-x_{\Ical^*}^l\|_2.
	 \end{aligned}\]
	Therefore, we only have to prove $S_k$ also has the same upper bound as in \eqref{linear.1} and  \eqref{linear.2}. 
	
	To derive the upper bound for $S_k$,  we continue with   \eqref{sum.x.bound}.  
	For any $k > \bar k$ with $\bar k$ defined in the proof of Lemma \ref{lem.acc.gc2.conv}, 
	we can use the same argument of deriving \eqref{sum.x.bound} to have 
	\begin{equation}\label{eq:rate1}
	\begin{aligned}
	S_k 
	&\leq \frac{2D_1}{D_2}\phi(\psi(x_{\Ical^*}^k,x_{\Ical^*}^{k-1}, \delta_{\Ical^*}^k)-\zeta)+\frac{1}{2}\|x_{\Ical^*}^{k-1}-x_{\Ical^*}^{k-2}\|_2+\frac{1}{2} \|\delta_{\Ical^*}^{k-1}\|_1\\
	& = \frac{2D_1}{D_2}\phi(\psi(x_{\Ical^*}^k, x_{\Ical^*}^{k-1}, \delta_{\Ical^*}^k)-\zeta) + \frac{1}{2}(S_{k-2}-S_{k-1})+\frac{1}{2} \|\delta_{\Ical^*}^{k-1}\|_1.
	\end{aligned}
	\end{equation}
	By KL inequality with $\phi'(s)=c(1-\theta)s^{-\theta}$, for $k > \bar k$,
	\begin{equation}
	\label{eq:rate2}
	c(1-\theta)(\psi(x_{\Ical^*}^k, x_{\Ical^*}^{k-1}, \delta_{\Ical^*}^k)-\zeta)^{-\theta} \| \nabla \psi(x_{\Ical^*}^k,x_{\Ical^*}^{k-1},\delta_{\Ical^*}^{k}))\|_2 \geq 1.
	\end{equation}

	On the other hand, using Lemma \ref{lem.acc.gc2}(i) and the definition of $S_k$, we see that for all sufficiently large $k$,
	\begin{equation}
	\label{eq:rate3}
	\| \nabla \psi(x_{\Ical^*}^k,x_{\Ical^*}^{k-1},\delta_{\Ical^*}^{k}))\|_2 \leq D_1(S_{k-2}-S_k+\|\delta_{\Ical^*}^{k-1}\|_1-\|\delta_{\Ical^*}^{k}\|_1)
	\end{equation}
	Combining \eqref{eq:rate2} with \eqref{eq:rate3}, we have
	\[
	(\psi(x_{\Ical^*}^k,x_{\Ical^*}^{k-1},\delta_{\Ical^*}^k)-\zeta)^{\theta}\leq D_1  c(1-\theta) (S_{k-2}-S_k+\|\delta_{\Ical^*}^{k-1}\|_1-\|\delta_{\Ical^*}^{k}\|_1)
	\]
\end{proof}
Taking a power of $(1-\theta)/\theta$ to both sides of the above inequality and scaling both sides by $c$, we obtain that for all $k> \bar k$
\begin{equation}
\begin{aligned}
\phi(\psi(x_{\Ical^*}^k,x_{\Ical^*}^{k-1}, \delta_{\Ical^*}^k)-\zeta)&=c[\psi(x_{\Ical^*}^k,x_{\Ical^*}^{k-1})-\zeta]^{1-\theta}\\
&\leq c  \Big[D_1    c(1-\theta) (S_{k-2}-S_k+\|\delta_{\Ical^*}^{k-1}\|_1-\|\delta_{\Ical^*}^{k}\|_1)\Big]^{\frac{1-\theta}{\theta}}\\
&\leq c  \Big[D_1    c(1-\theta) (S_{k-2}-S_k+\|\delta_{\Ical^*}^{k-1}\|_1)\Big]^{\frac{1-\theta}{\theta}},
\end{aligned}
\end{equation}
which combined with   \eqref{eq:rate1} yields 
\begin{equation}\label{eq:rate5}
\begin{aligned}
S_k &\leq C_1[S_{k-2}-S_k+\|\delta_{\Ical^*}^{k-1}\|_1]^{\frac{1-\theta}{\theta}} + \frac{1}{2}[S_{k-2}-S_k+\|\delta_{\Ical^*}^{k-1}\|_1]
\end{aligned}
\end{equation}
where $C_1=\frac{2D_1c}{D_2} \left(D_1\cdot c(1-\theta)\right)^{\frac{1-\theta}{\theta}}$.  
It follows that  
\begin{equation}
\begin{aligned}
&S_k + \frac{\sqrt{\mu}}{1-\mu}\|\delta_{\Ical^*}^k\|_1 \\
&\leq   C_1[S_{k-2}-S_k+\|\delta_{\Ical^*}^{k-1}\|_1]^{\frac{1-\theta}{\theta}} + \frac{1}{2}[S_{k-2}-S_k+\|\delta_{\Ical^*}^{k-1}\|_1]+ \frac{\sqrt{\mu}}{1-\mu}\|\delta_{\Ical^*}^k\|_1\\
& \leq  C_1[S_{k-2}-S_k+\|\delta_{\Ical^*}^{k-1}\|_1]^{\frac{1-\theta}{\theta}} + \frac{1}{2}[S_{k-2}-S_k+\|\delta_{\Ical^*}^{k-1}\|_1]+ \frac{\mu}{1-\mu}\|\delta_{\Ical^*}^{k-1}\|_1\\
&\leq C_1 [S_{k-2}-S_k+\|\delta_{\Ical^*}^{k-1}\|_1]^{\frac{1-\theta}{\theta}}+ C_2[S_{k-2}-S_k+\|\delta_{\Ical^*}^{k-1}\|_1],
\end{aligned}
\end{equation}
with  
$
C_2:=\frac{1}{2}+\frac{\mu}{1-\mu}.
$

For part (ii), $\theta \in (0, \frac{1}{2}]$.   Notice that 
\[\frac{1-\theta}{\theta}\geq 1\ \text{ and }\  S_{k-2}-S_k+\|\delta_{\Ical^*}^{k-1}\|_1 \to 0.\]
Hence, there exists $k_1\ge \bar k$ such that for any $k\ge k_1$
\[ [S_{k-2}-S_k+\|\delta_{\Ical^*}^{k-1}\|_1]^{\frac{1-\theta}{\theta}}  \le  [S_{k-2}-S_k+\|\delta_{\Ical^*}^{k-1}\|_1].\]
This,  combined with \eqref{eq:rate5}, yields 
\[ S_k + \frac{\sqrt{\mu}}{1-\mu}\|\delta^k_{\Ical^*}\|_1 \le (C_1+C_2) [S_{k-2}-S_k+\|\delta_{\Ical^*}^{k-1}\|_1]\]
for any $k\ge k_1$. By $\delta_i^{k}\leq \sqrt{\mu}\delta_i^{k-1}$, we know that for $i\in {\Ical^*}$,
\begin{align} \label{eq:delta_rel}
\delta_i^{k-1}\leq \frac{\sqrt{\mu}}{1-\mu}(\delta_i^{k-2}-\delta_i^k).
\end{align}
It follows that 
\[\begin{aligned}
S_k + \frac{\sqrt{\mu}}{1-\mu}\|\delta^k\|_1  &\le (C_1+C_2)\bigg[\big(S_{k-2}+\frac{\sqrt{\mu}}{1-\mu}\|\delta_{\Ical^*}^{k-2}\|_1\big)-\big(S_{k}+\frac{\sqrt{\mu}}{1-\mu}\|\delta_{\Ical^*}^{k}\|_1
\big)\bigg]
\end{aligned}
\]
So, 
\[\begin{aligned}
S_k + \frac{\sqrt{\mu}}{1-\mu}\|\delta_{\Ical^*}^k\|_1  \le &\frac{C_1+C_2}{C_1+C_2+1}\bigg[S_{k-2}+\frac{\sqrt{\mu}}{1-\mu}\|\delta_{\Ical^*}^{k-2}\|_1\bigg] \\
\le & \left( \frac{C_1+C_2}{C_1+C_2+1} \right)^{ \lfloor \frac{k}{2} \rfloor } \bigg[S^{k\bmod 2}+\frac{\sqrt{\mu}}{1-\mu}\|\delta_{\Ical^*}^{k\bmod 2}\|_1\bigg]\\
\le & \left( \frac{C_1+C_2}{C_1+C_2+1} \right)^{ \frac{k-1}{2} } \bigg[S^0+\frac{\sqrt{\mu}}{1-\mu}\|\delta_{\Ical^*}^0\|_1\bigg].
\end{aligned}
\]
Hence, we have 
\[
S_k+\frac{\sqrt{\mu}}{1-\mu}\|\delta_{\Ical^*}^k\|_1\leq \bigg(\sqrt{\frac{C_1+C_2}{C_1+C_2+1}}\bigg)^{k-1}(S_1+\frac{\sqrt{\mu}}{1-\mu}\|\delta_{\Ical^*}^1\|_1)
\]
Therefore, for any $k\ge k_1$,
\[ \|x^k-x^*\|_2 \le S_k \le c_1 \gamma^{k-1} -c_2 \|\delta_{\Ical^*}^{k}\|_1 \]
with 
\[ c_1= (S_1+\frac{\sqrt{\mu}}{1-\mu}\|\delta_{\Ical^*}^1\|),  \gamma=\sqrt{\frac{C_1+C_2}{C_1+C_2+1}}\ \text{ and }\ c_2= \frac{\sqrt{\mu}}{1-\mu},\] 
completing the proof of (ii).

For part (iii), $\theta \in (\frac{1}{2},1)$.   Notice that 
\[\frac{1-\theta}{\theta}< 1\ \text{ and }\  S_{k-2}-S_k+\|\delta_{\Ical^*}^{k-1}\|_1 \to 0.\]
Hence, there exists $k_2\ge \bar k$ such that for any $k\ge k_2$
\[[S_{k-2}-S_k+\|\delta_{\Ical^*}^{k-1}\|_1] \le [S_{k-2}-S_k+\|\delta_{\Ical^*}^{k-1}\|_1]^{\frac{1-\theta}{\theta}}.\]
This,  combined with \eqref{eq:rate5}, yields 
\[ S_k + \frac{\sqrt{\mu}}{1-\mu}\|\delta_{\Ical^*}^k\|_1 \le (C_1+C_2) [S_{k-2}-S_k+\|\delta_{\Ical^*}^{k-1}\|_1]^{\frac{1-\theta}{\theta}}\]
for any $k\ge k_1$.  Combining with \eqref{eq:delta_rel} yields 
\begin{equation}
S_k+ \frac{\sqrt{\mu}}{1-\mu}\|\delta_{\Ical^*}^k\|_1 \leq (C_1+C_2) [S_{k-2}+\frac{\sqrt{\mu}}{1-\mu}\|\delta_{\Ical^*}^{k-2}\|_1-(S_k+\frac{\sqrt{\mu}}{1-\mu}\|\delta_{\Ical^*}^{k}\|_1)]^{\frac{1-\theta}{\theta}}
\end{equation}
Raising to a power of $\frac{\theta}{1-\theta}$ to both side of the above equation, we see further that for any $k>k_2$
\begin{equation}
(S_k+\frac{\sqrt{\mu}}{1-\mu}\|\delta_{\Ical^*}^k\|_1)^{\frac{\theta}{1-\theta}} = C_3[S_{k-2}+\frac{\sqrt{\mu}}{1-\mu}\|\delta_{\Ical^*}^{k-2}\|_1-(S_k+\frac{\sqrt{\mu}}{1-\mu}\|\delta_{\Ical^*}^{k}\|_1)]
\end{equation}
with $C_3:=(C_1+C_2)^{\frac{\theta}{1-\theta}}$.

Consider the ``even'' subsequence of $\{k_2, k_2+1, \ldots\}$ and define $\{\Delta_t\}_{t\ge \lceil k_2/2\rceil}$ with 
$\Delta_t:= S_{2t}+\frac{\sqrt{\mu}}{1-\mu}\|\delta_{\Ical^*}^{2t}\|_1$. Then for all $t\ge \lceil k_2/2 \rceil$, we have
\[
\Delta_t^{\frac{\theta}{1-\theta}} \leq C_3(\Delta_{t-1}-\Delta_t)
\]
Proceeding as in the proof of \cite[Theorem 2]{Attouch2009} (starting from  \cite[Equation (13)]{Attouch2009}).
Define $h:(0,+\infty)\to \mathbb{R}$ by $h(s)=s^{-\frac{\theta}{1-\theta}}$ and let $R\in (1,+\infty)$. Take $k\geq N_1$ and assume first that $h(\Delta_k)\leq Rh(\Delta_{k-1})$. By rewriting above formula as
\[
1\leq \frac{C_3(\Delta_{k-1}-\Delta_{k})}{\Delta^{\frac{\theta}{1-\theta}}_{k}},
\]

we obtain that 
\[\begin{aligned}
1 &\leq C_3(\Delta_{k-1}-\Delta_{k})h(\Delta_k)\\
&\leq RC_3(\Delta_{k-1}-\Delta_{k})h(\Delta_{k-1})\\
&\leq RC_3\int_{\Delta_k}^{\Delta_{k-1}}h(s)ds\\
&\leq RC_3\frac{1-\theta}{1-2\theta}[\Delta_{k-1}^{\frac{1-2\theta}{1-\theta}}-\Delta_{k}^{\frac{1-2\theta}{1-\theta}}].
\end{aligned}
\]

Thus if we set $\mu=\frac{2\theta-1}{(1-\theta)RC_3}>0$ and $\nu=\frac{1-2\theta}{1-\theta}<0$ one obtains that 
\begin{equation}
\label{eq:delta2}
0<\mu \leq \Delta_{k}^{\nu}-\Delta_{k-1}^{\nu}.
\end{equation}

Assume now that $h(\Delta_k)>Rh(\Delta_k)$ and set $q=(\frac{1}{R})^{\frac{1-\theta}{\theta}} \in (0,1)$. It follows immediately that $\Delta_k\leq q\Delta_{k-1}$ and furthermore  (recalling that $\nu$ is negative) we have
\[\begin{aligned}
\Delta_k^{\nu}&\geq q^{\nu}\Delta_{k-1}^{\nu}\\
\Delta_k^{\nu}-\Delta_{k-1}^{\nu} &\geq (q^{\nu}-1)\Delta_{k-1}^{\nu}.
\end{aligned}
\]
Since $q^{\nu}-1>0$ and $\Delta_{p}\to 0^+$ as $p\to +\infty$, there exists $\bar{\mu}>0$ such that $(q^{\mu}-1)\delta_{p-1}^{\nu}>\bar{\mu}$ for all $p\geq N_1$. Therefore we obtain that 
\begin{equation}
\label{eq:delta3}
\Delta_{k}^{\nu}-\Delta_{k-1}^{\nu}\geq \bar{\mu}.
\end{equation}
If we set $\hat{\mu}=\min\{\mu, \bar{\mu}\}>0$, one can combine (\ref{eq:delta2}) and (\ref{eq:delta3}) to obtain that 
\[
\Delta_k^{\nu}-\Delta_{k-1}^{\nu}\geq \hat{\mu}>0
\]
for all $k\geq N_1$. By summing those inequalities from $N_1$ to some $t$ greater than $N_1$ we obtain that $\Delta_t^{\nu}-\Delta_{N_1}^{\nu}\geq \hat{\mu}(N-N_1)$ and consequently follows from
\begin{equation}
\label{power.delta}
\Delta_t\leq [\Delta_{N_1}^{\nu}+\hat{\mu}(t-N_1)]^{1/{\nu}}\leq C_4t^{-\frac{1-\theta}{2\theta-1}},
\end{equation}
for some $C_4>0$.

As for the ``odd'' subsequence of $\{k_2, k_2+1, \ldots\}$, we can define $\{\Delta_t\}_{t\ge \lceil k_2/2\rceil}$ with 
$\Delta_t:= S_{2t+1}+\frac{\sqrt{\mu}}{1-\mu}\|\delta_{\Ical^*}^{2t+1}\|_1$ and then can still show that \eqref{power.delta} holds true. 

Therefore,  for all sufficiently large and even number  $k$,  
\[ \|x_{\Ical^*}^k-x_{\Ical^*}^*\|_2 \leq S_k =\Delta_{\frac{k}{2}}-\frac{\sqrt{\mu}}{1-\mu}\|\delta_{\Ical^*}^{k}\|_1 \leq 2^{{\frac{1-\theta}{2\theta-1}} } C_4 k^{-\frac{1-\theta}{2\theta-1}}-\frac{\sqrt{\mu}}{1-\mu}\|\delta_{\Ical^*}^{k}\|_1.\]  
For all sufficiently large and odd number  $k$,  
\[ \|x_{\Ical^*}^k-x_{\Ical^*}^*\|_2 \leq S_k =\Delta_{\frac{k-1}{2}}-\frac{\sqrt{\mu}}{1-\mu}\|\delta_{\Ical^*}^{k}\|_1 \leq 2^{{\frac{1-\theta}{2\theta-1}} } C_4 (k-1)^{-\frac{1-\theta}{2\theta-1}}-\frac{\sqrt{\mu}}{1-\mu}\|\delta_{\Ical^*}^{k}\|_1.\]  
Overall, we have 
\[ \|x^k-x^*\|=\|x_{\Ical^*}^k-x_{\Ical^*}^*\|_2 \le c_3 k^{{\frac{1-\theta}{2\theta-1}} } - c_4 \|\delta_{\Ical^*}^{k}\|_1\]
for sufficiently large $k$, where 
\[c_3:=2^{{\frac{1-\theta}{2\theta-1}} } C_4\ \text{ and }\  c_4:= \frac{\sqrt{\mu}}{1-\mu}.\]
This completes the proof.

\section{Numerical results}
In this section, we perform sparse signal recovery experiments (similar to  \cite{yu2019iteratively,zeng2016sparse, figueiredo2007gradient,wen2018proximal}) to study the behaviors of Algorithm \ref{alg.acc}. The goal of the experiments is to reconstruct a length $n$ sparse signal $x$ from $m$ observations via its incomplete measurements with $m<n$. For this purpose, we first generate an $m\times n$ matrix $A$ with i.i.d. standard Gaussian entries and orthonormalizing the rows. We then set $y=Ax_{true} + \varepsilon$, where the origin signal $x_{true}$ contains $K$ randomly placed $\pm1$ spikes and  $\varepsilon\in \mathbb{R}^{m}$ has i.i.d standard Gaussian entries with variance $\sigma^2=10^{-4}$. Hence, the objective  has  $f(x)=\frac{1}{2}\|Ax-y\|^2_2$, 
where $A\in \mathbb{R}^{m\times n}, y\in \mathbb{R}^m$ and $x\in \mathbb{R}^n$. 

In our experiments, we compare the performances of the proposed algorithm EIRL1 with  IRL1, iteratively reweighted $\ell_2$ algorithm (IRL2)  and  the iterative jumping thresholding (IJT) algorithm  \citep{zeng2016sparse} for solving $\ell_p$-norm optimization problem and study the role of $\alpha$. All algorithms start from a random Gaussian vector $x^0$ with mean 0 and variance 1, and have same termination criterion that the number of iteration exceeds the limit or 
\[
\frac{\|x^k-x^{k-1}\|}{\|x^k\|}\leq \texttt{opttol},
\]
where \texttt{opttol} is a small parameter, and the default is $10^{-6}$.

\subsection{Comparison with other types of algorithms}
We compare the computational efficiency of algorithms for solving $\ell_p$ regularization problems with   $p=2/3, 1/2, 6/7$.  We fix the matrix size $(m,n)=(2048,4096)$, the degrees of sparseness $K=200$, $\lambda = 0.05, \mu=0.9, \beta = 1,   \epsilon^0 =1$ and   $\alpha =0.9$ for   
EIRL1.  For each $p$, we generate 50 random data sets $(A, x_{true}, y)$ and compute average MSE with respect to $x_{true}$, that is, $\text{MSE} = \tfrac{1}{n}\|x^k-x_{true}\|^2_2$. Figure \ref{fig.IRLIJT} depicts  the average MSE over the number of iterations, which shows  that EIRL1 converges faster than other methods. In addition, it is worth noticing that IRL algorithms have the ability to solve $\ell_p$-norm regularization problem with any $0<p<1$, while IJT  is limited to those $p$ such that the subproblems possess an explicit solution.
\begin{figure}[htbp] 
	\centering 
	\subfigure[$p=\frac{1}{2}$]{ 
		\includegraphics[width=2.2in]{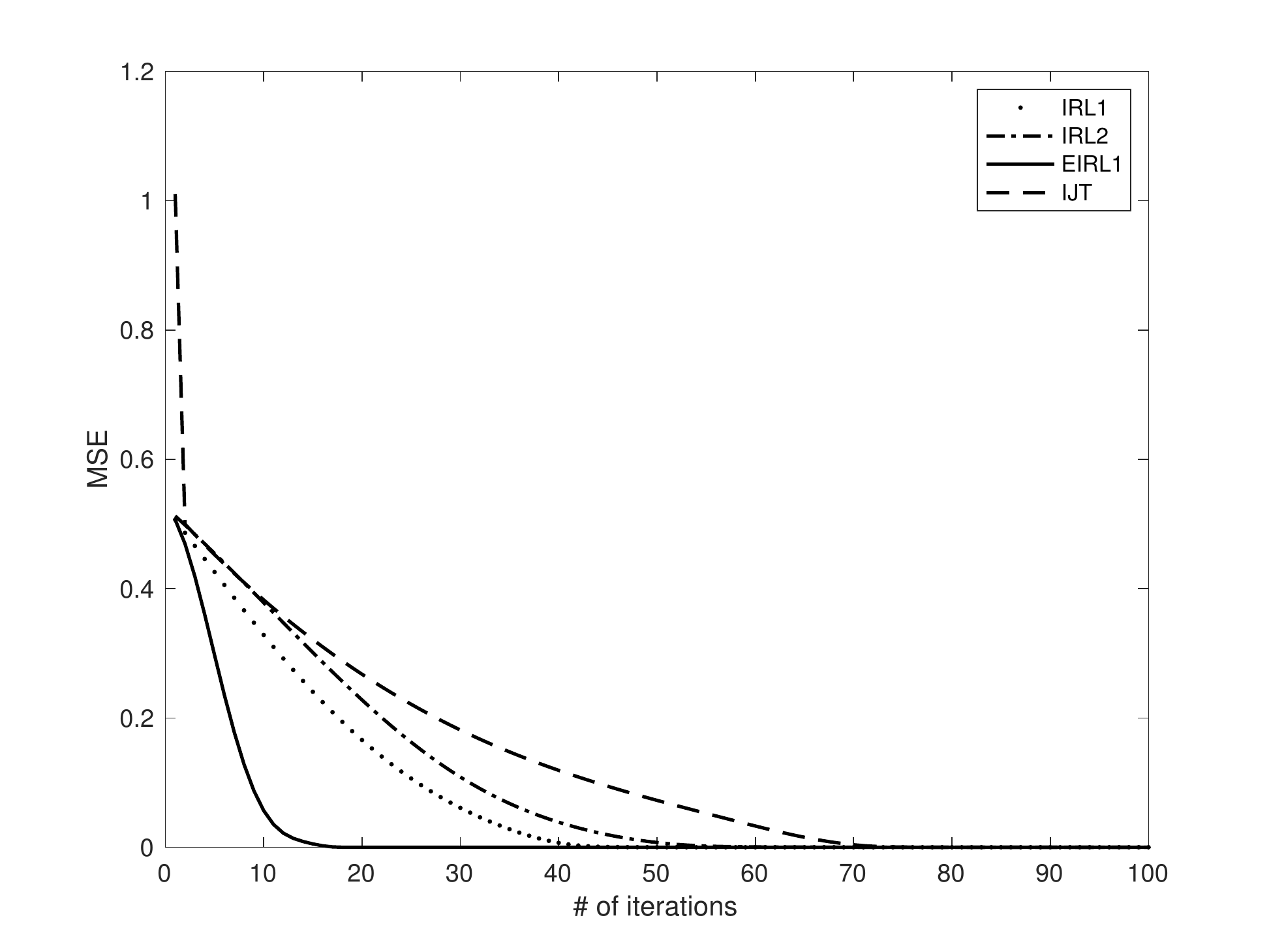} 
	} 
	\subfigure[$p=\frac{2}{3}$]{ 
		\includegraphics[width=2.2in]{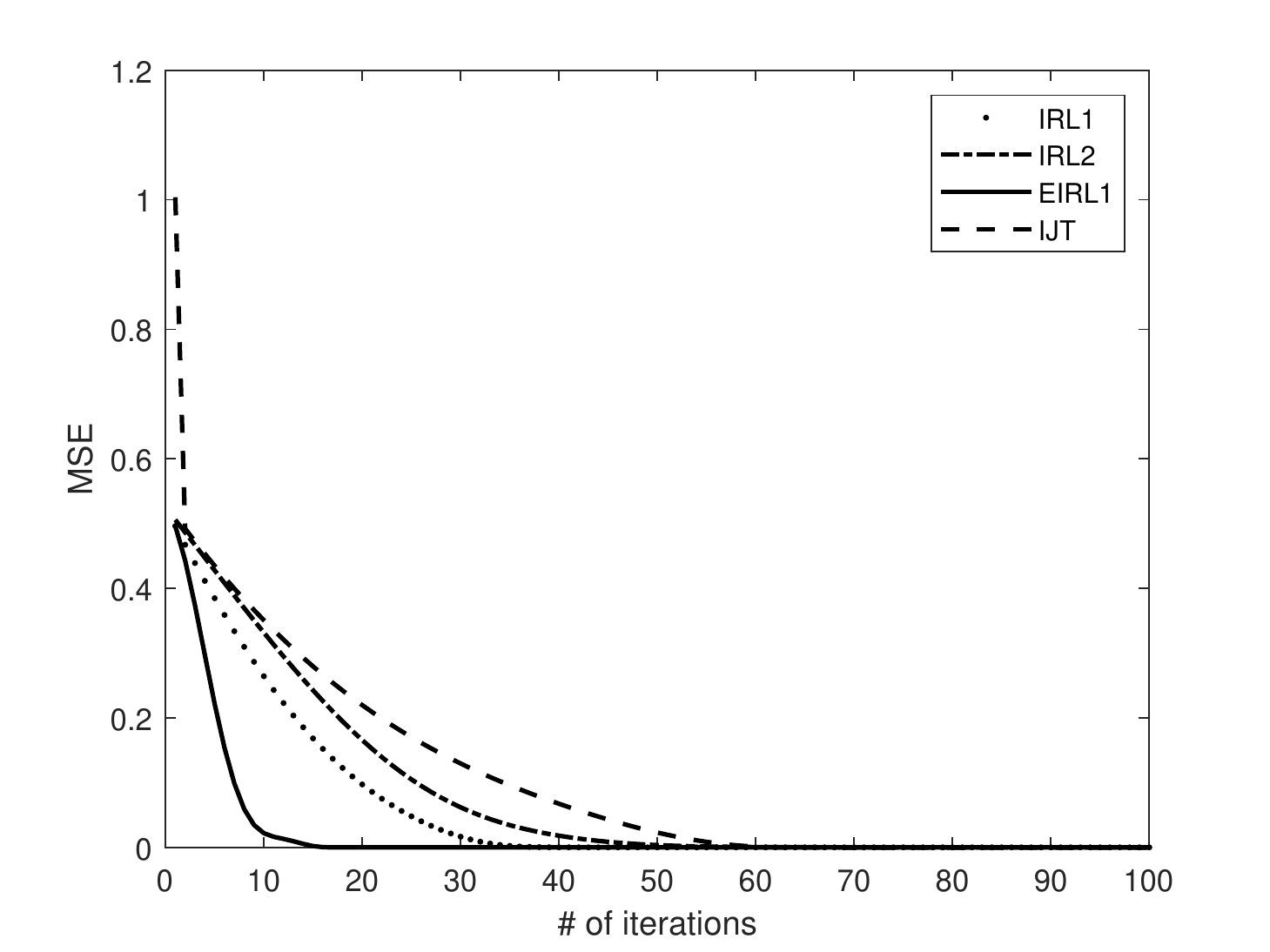} 
	} 
	\subfigure[$p=\frac{6}{7}$]{ 
		\includegraphics[width=2.2in]{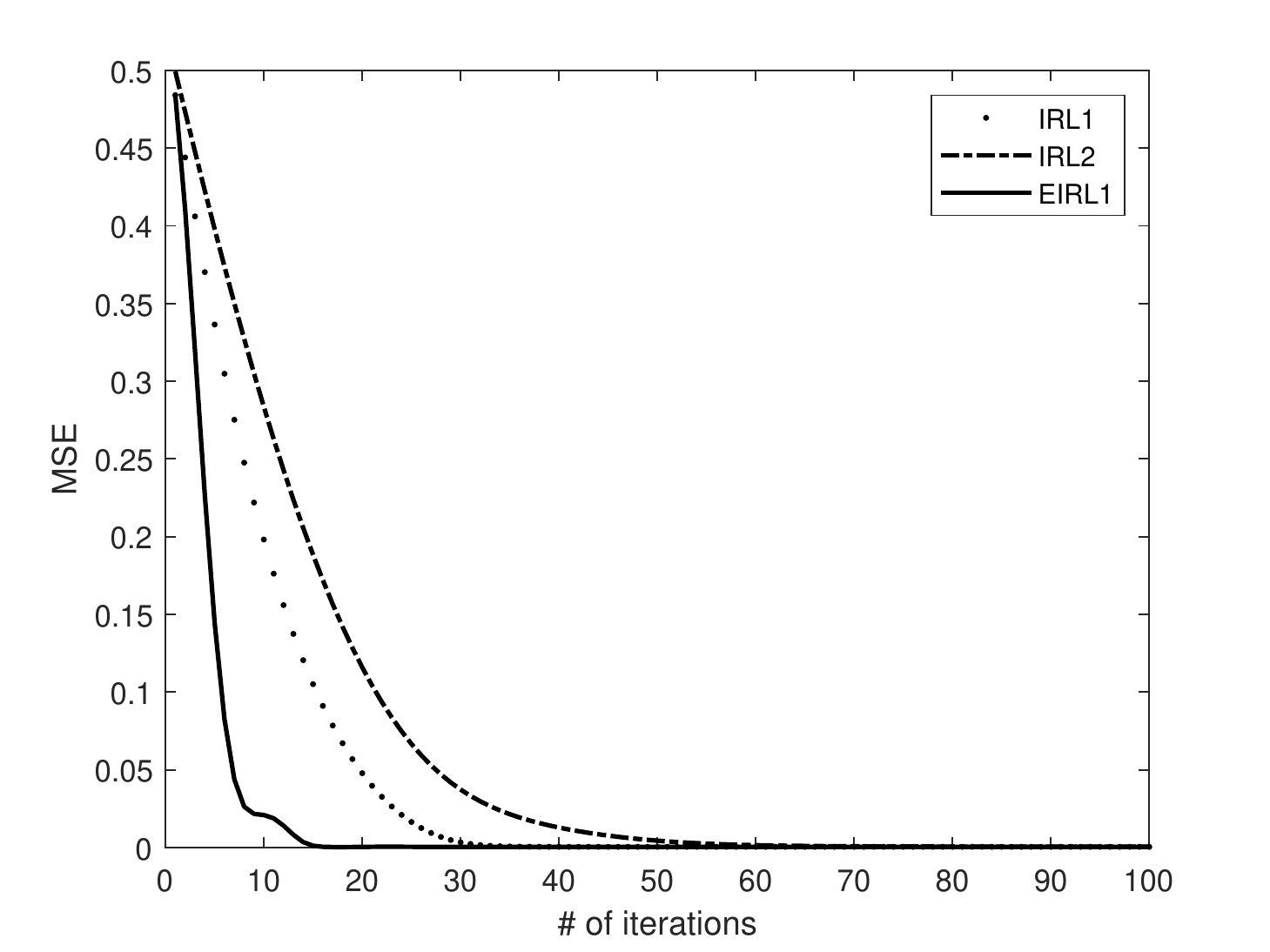} 
	} 
	\caption{Average MSE versus the number of iterations for IRL1, IRL2 and IJT algorithms. For $p=1/2, p=2/3$ and $p=6/7$, EIRL1 converges faster than other algorithms. For $p=6/7$, there is no explicit subproblem solution for IJT algorithm, therefore it is excluded from the figure.} 
	\label{fig.IRLIJT}
\end{figure}

\subsection{The role  of $\alpha$}
Since EIRL1 involves a parameter $\alpha$,  a  common question is about the selection of $\alpha$. 
Therefore, we test  EIRL1 with different $\alpha$  in different situations with sparsity varying from small to large  in our last experimental setting. 
In particular, we set  $K= 200, 500, 800$ for $x_{true}$, representing the situations with low, medium and high sparsity, respectively.  For each $K$, we test EIRL1 with 5 different values of $\alpha=0, 0.3, 0.5, 0.7, 0.9$. For each $\alpha$, we generate 50 random data sets $(A,x_{true},y)$, compute the average MSE versus iterates, and record the number of nonzero components in the found solution. We plot the evolution of MSE and the  box-plots about the number of nonzero components with different $\alpha$ in Figure \ref{fig.twostep}. It shows that larger $\alpha$ has faster  convergence and can find sparser solutions.
\begin{figure}[htbp] 
	\centering 
	\subfigure[MSE for $K=200$]{ 
		\includegraphics[width=2.2in]{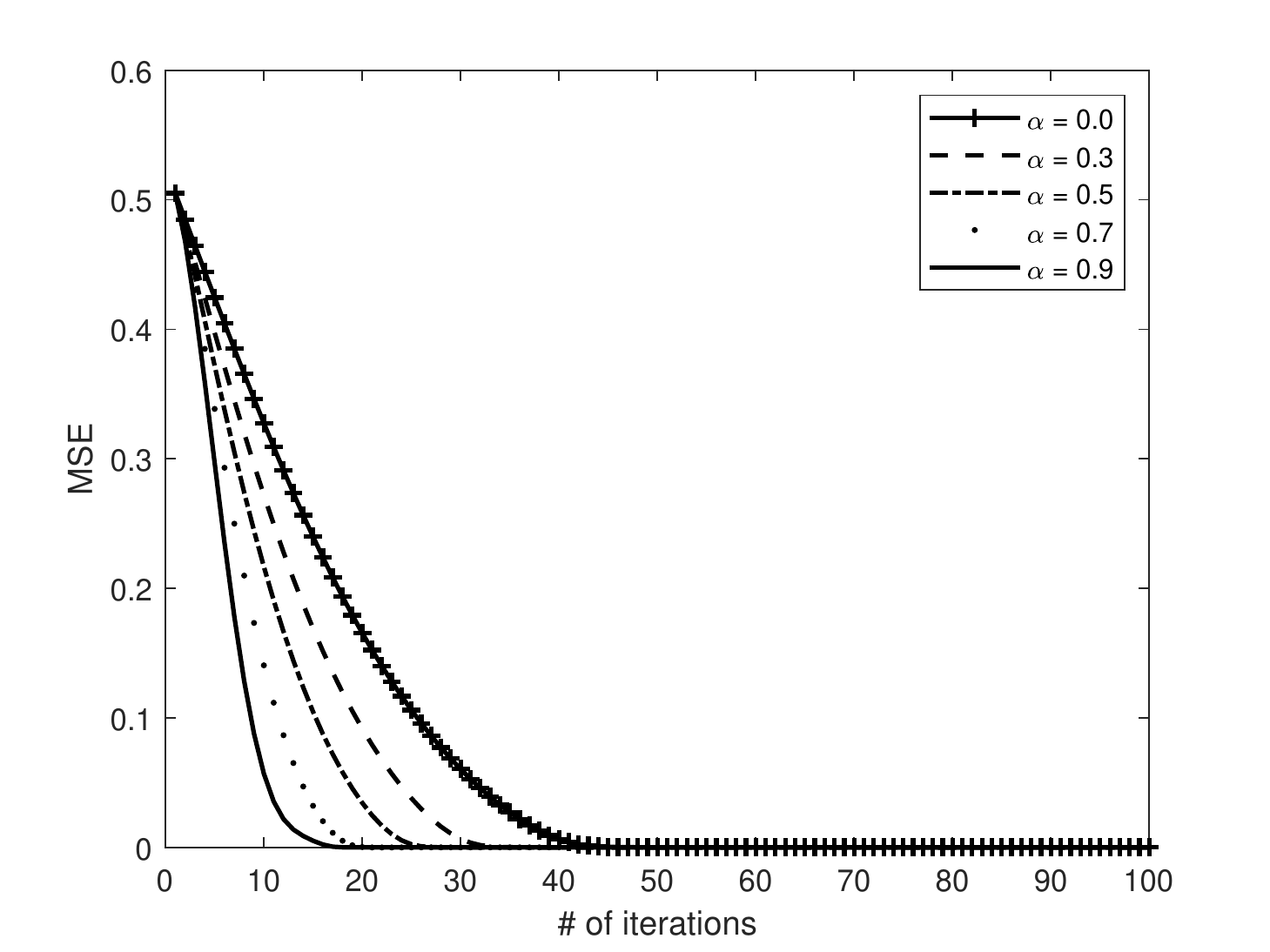} 
	} 
	\subfigure[Box-plot for $K=200$]{ 
		\includegraphics[width=2.2in]{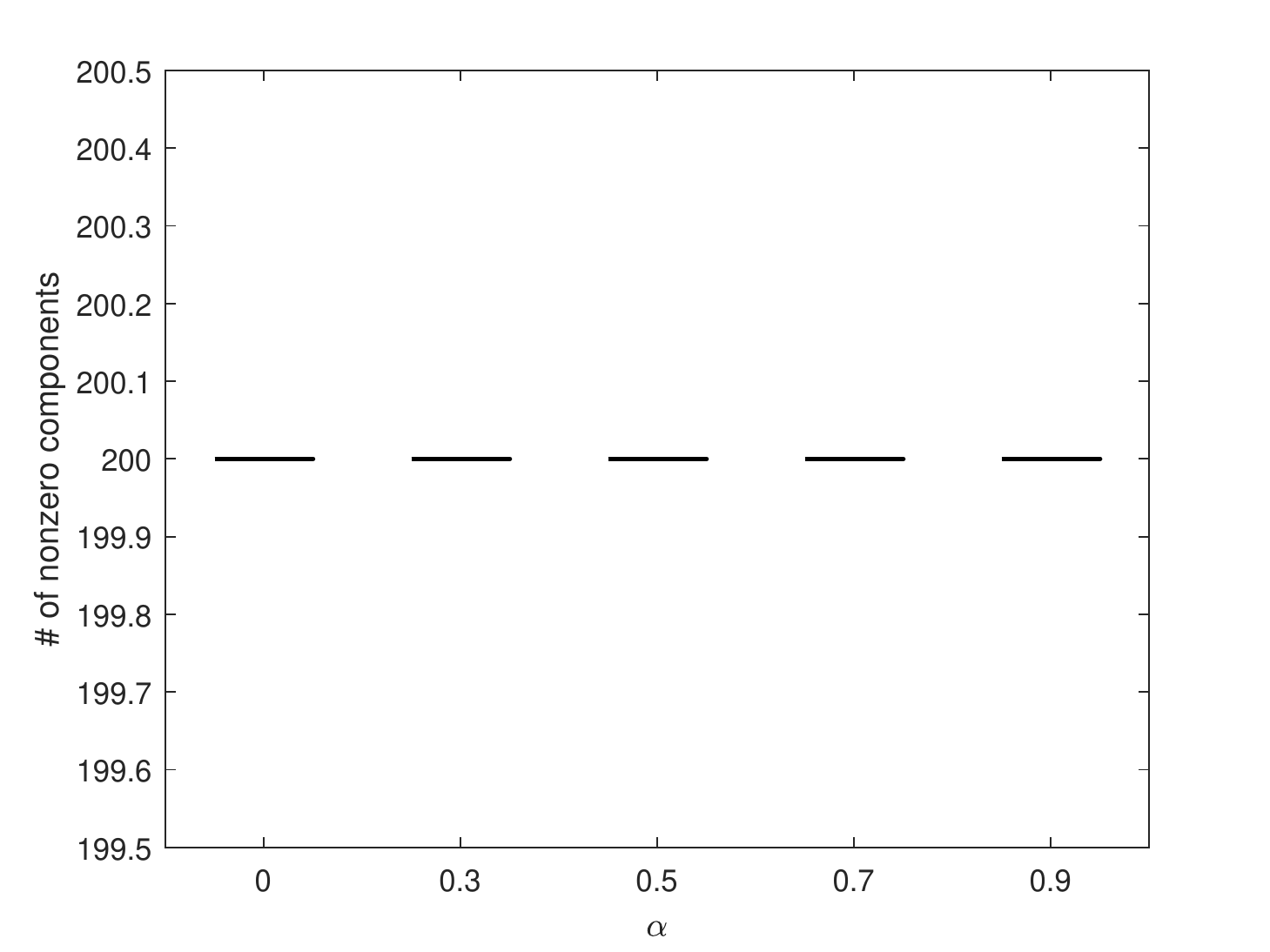} 
	} 
	\subfigure[MSE for $K=500$]{ 
		\includegraphics[width=2.2in]{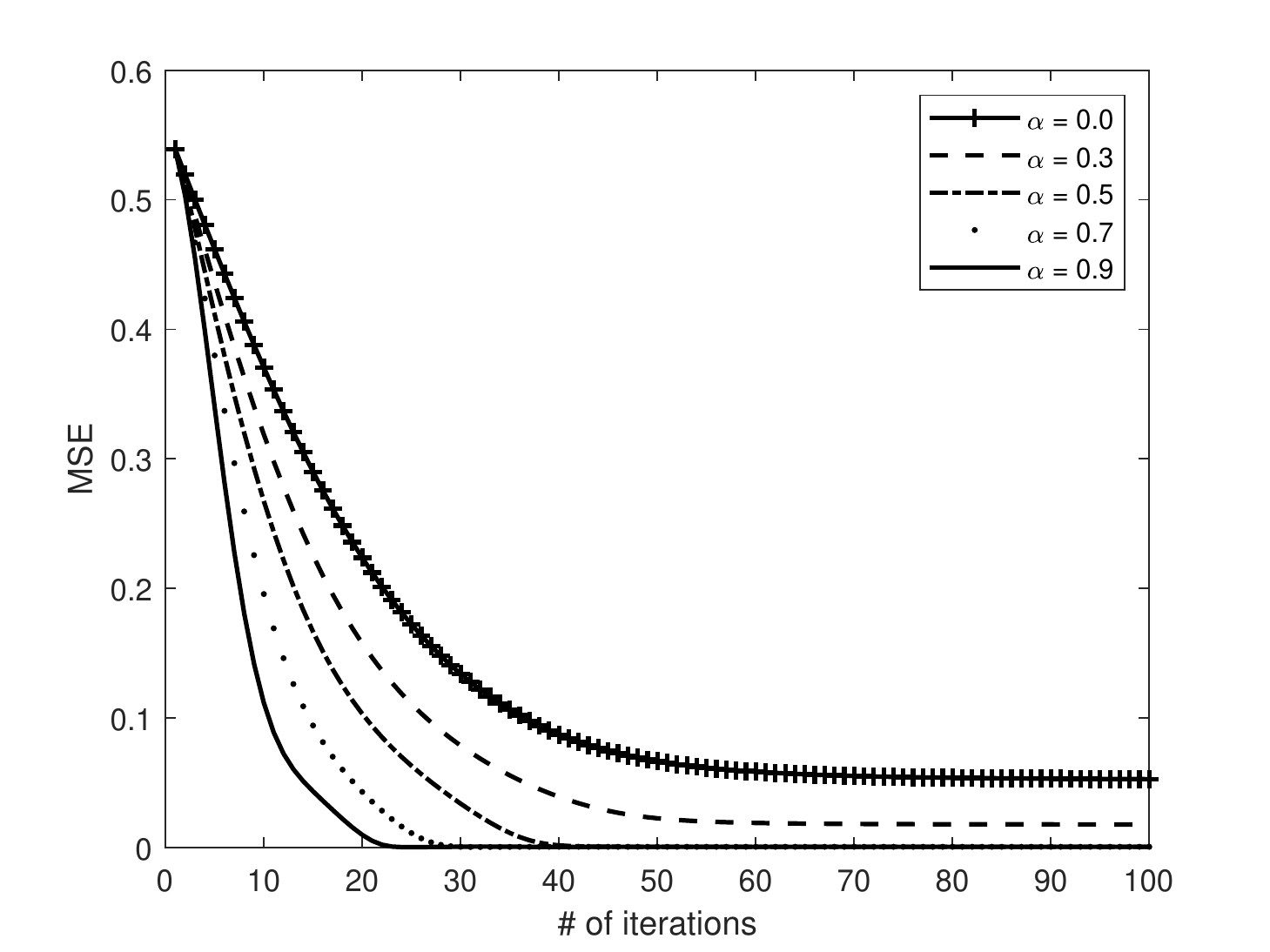} 
	} 
	\subfigure[Box-plot for $K=500$]{ 
		\includegraphics[width=2.2in]{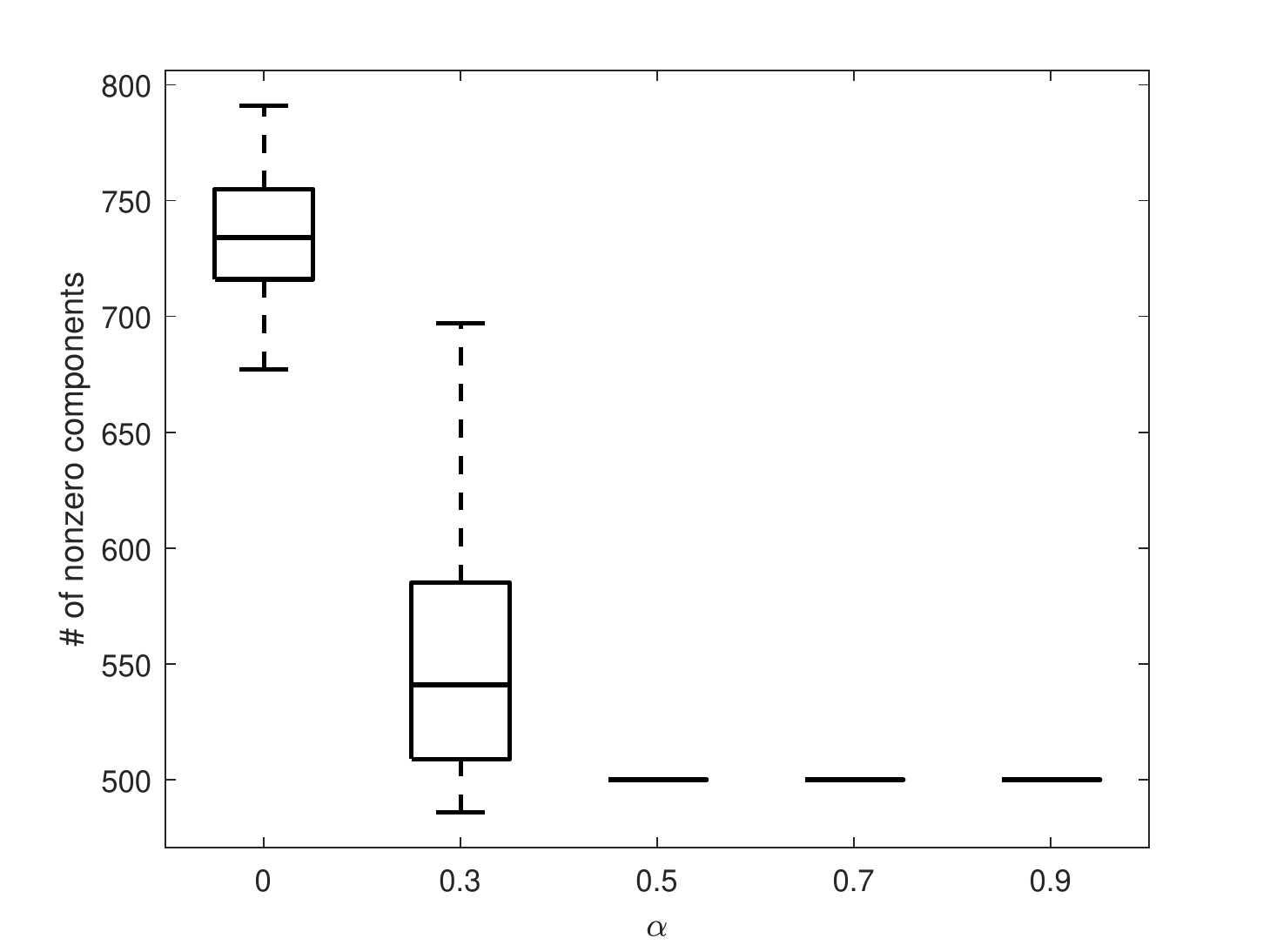} 
	} 
	\subfigure[MSE for $K=800$]{ 
		\includegraphics[width=2.2in]{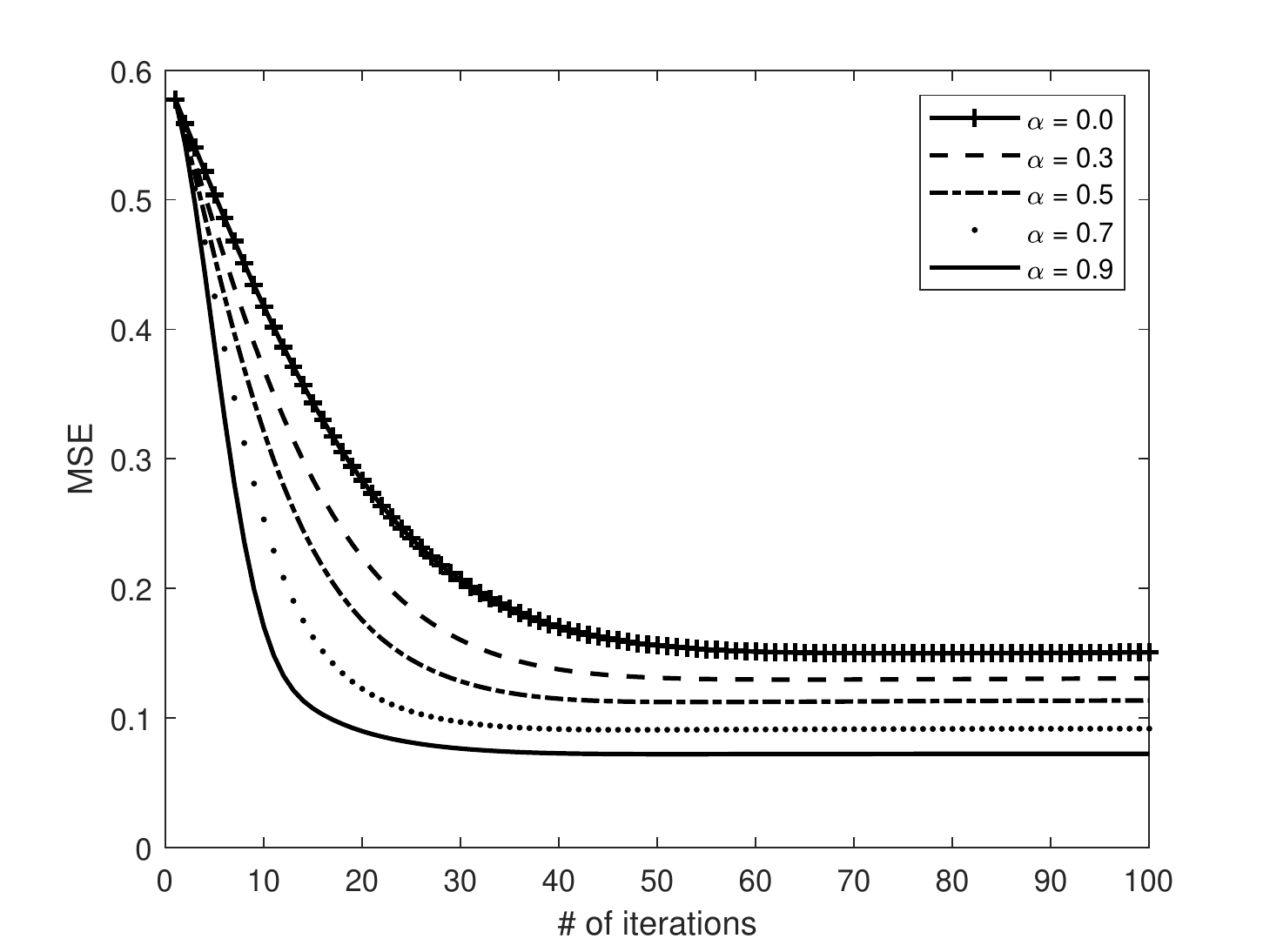} 
	} 
	\subfigure[Box-plot for $K=800$]{ 
		\includegraphics[width=2.2in]{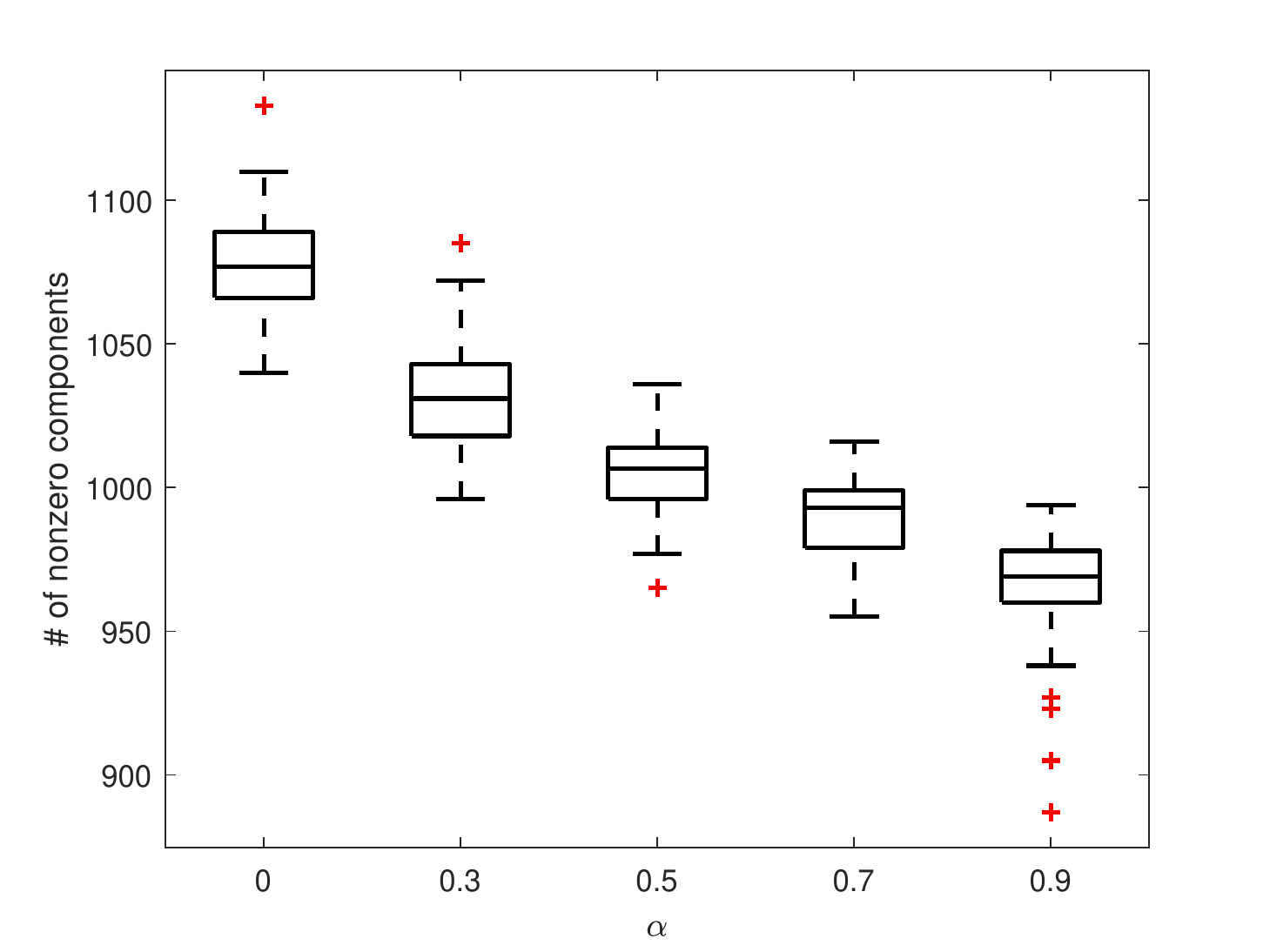} 
	} 
	\caption{The performance for different $\alpha$ when  $K=200, 500, 800$. It presents that larger $\alpha$ has better performance. 
	} 
	\label{fig.twostep}
\end{figure}

Notice that we require $0\le \alpha <1$ to guarantee the convergence of 
EIRL1. To find the optimal selection of $\alpha$,  we test more carefully with larger values $\alpha=$0.90, 0.93,0.95,0.97,0.99  with $K=800$ in Figure  \ref{fig.twostepproblem}.  It shows that $\alpha = 0.9$ outperforms other values of $\alpha$. 
\begin{figure}[htbp] 
	\centering 
	\subfigure[$\log(\text{MSE})$ for $K=800$]{ 
		\includegraphics[width=2.2in]{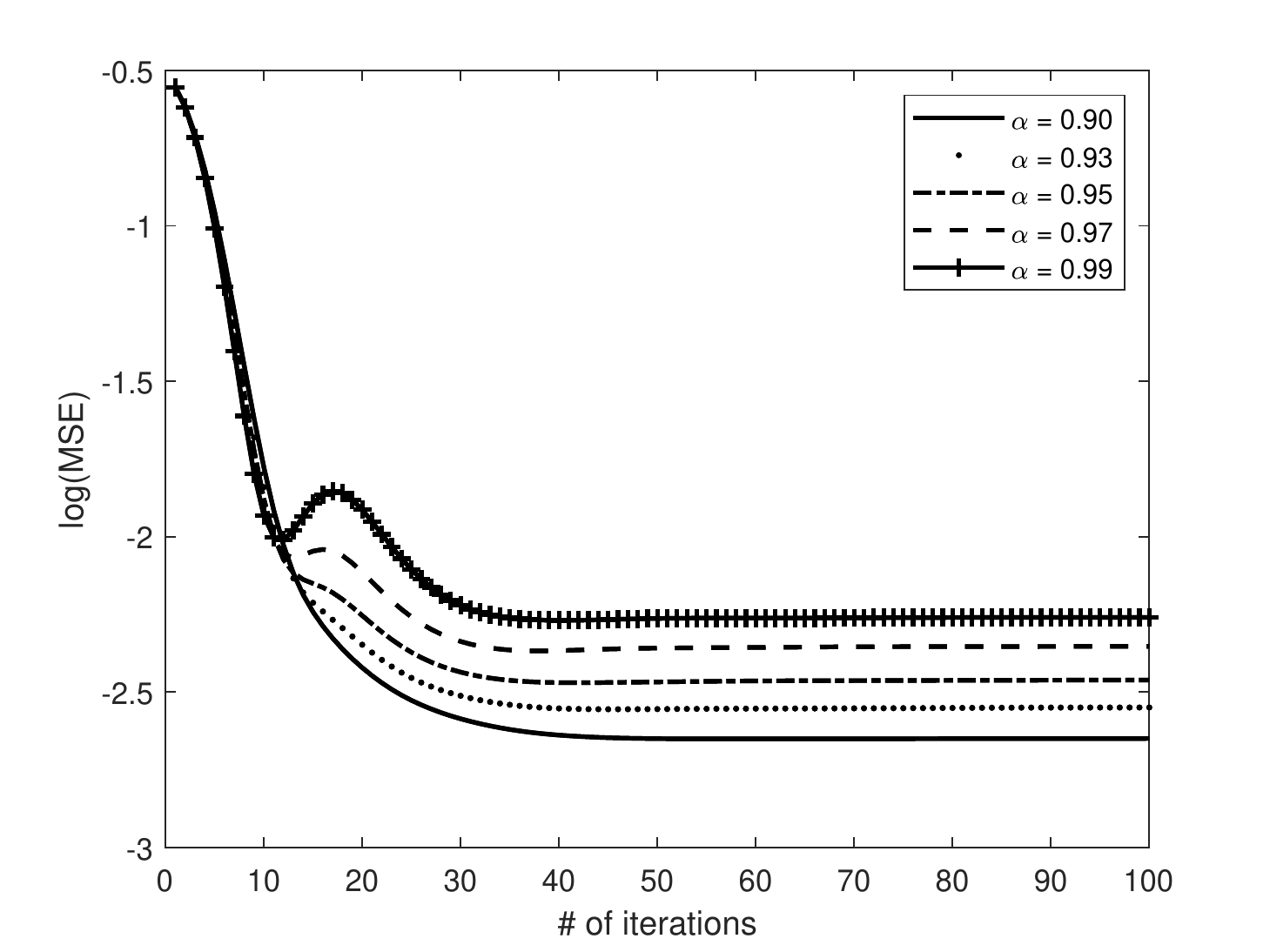} 
	} 
	\subfigure[Box-plot for $K=800$]{ 
		\includegraphics[width=2.2in]{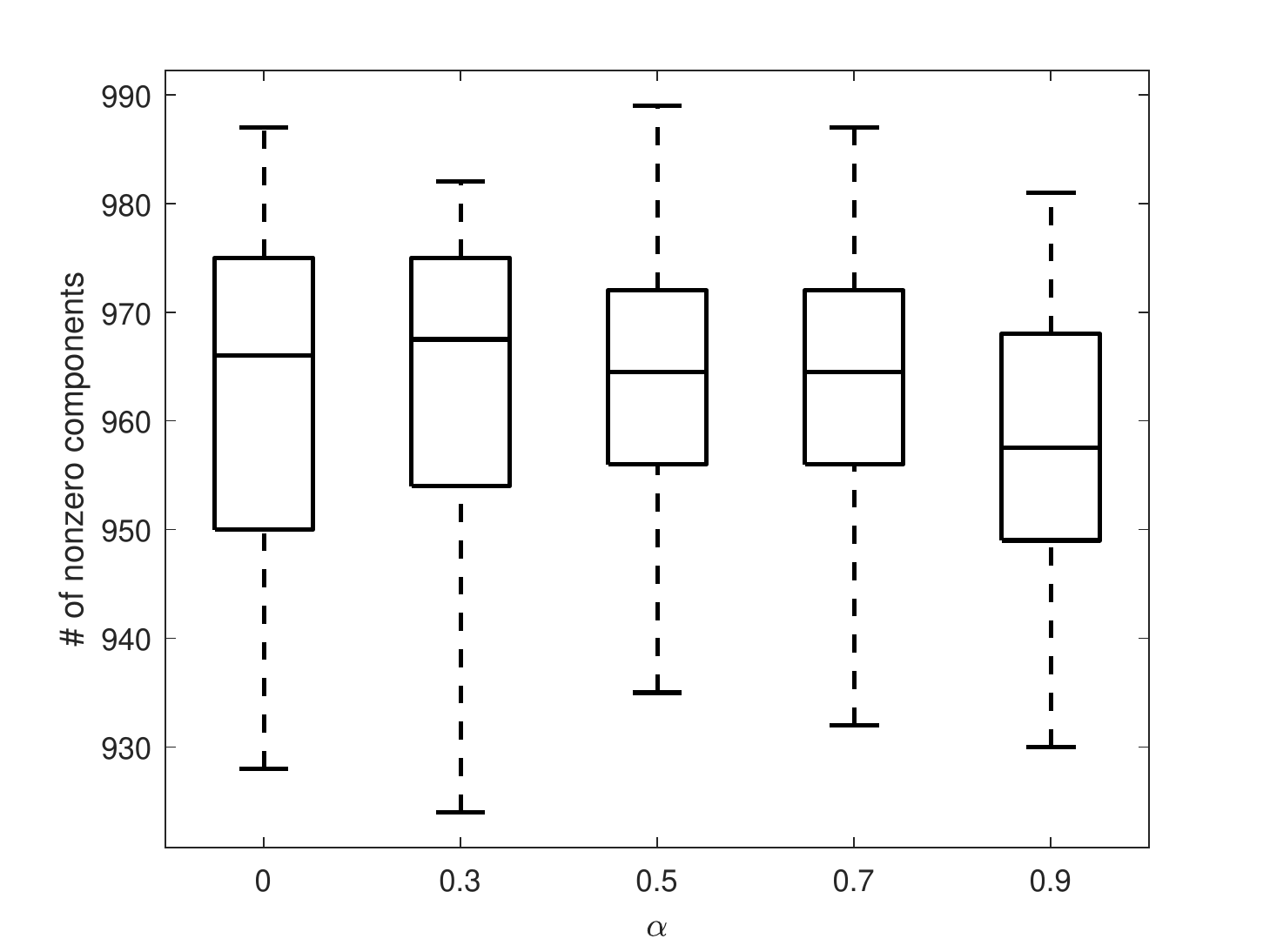} 
	} 
	\caption{The performance for different $\alpha$.  } 
	\label{fig.twostepproblem}
\end{figure}

In Figure  \ref{fig.nesterov}, we also test dynamically updating 
\[ \alpha^k=\frac{k-1}{k+2}\] for $k\geq 1$ and $\alpha^0 = 0$ for $K=800$, which is called Nesterov’s momentum coefficient  \cite{nesterov1983method}. We can see that this $\alpha^k$ have better performance than   $\alpha=0.7$ but slightly worse than 0.9. 
\begin{figure}[htbp] 
	\centering 
	\subfigure[$\log(\text{MSE})$ for $K=800$]{ 
		\includegraphics[width=2.2in]{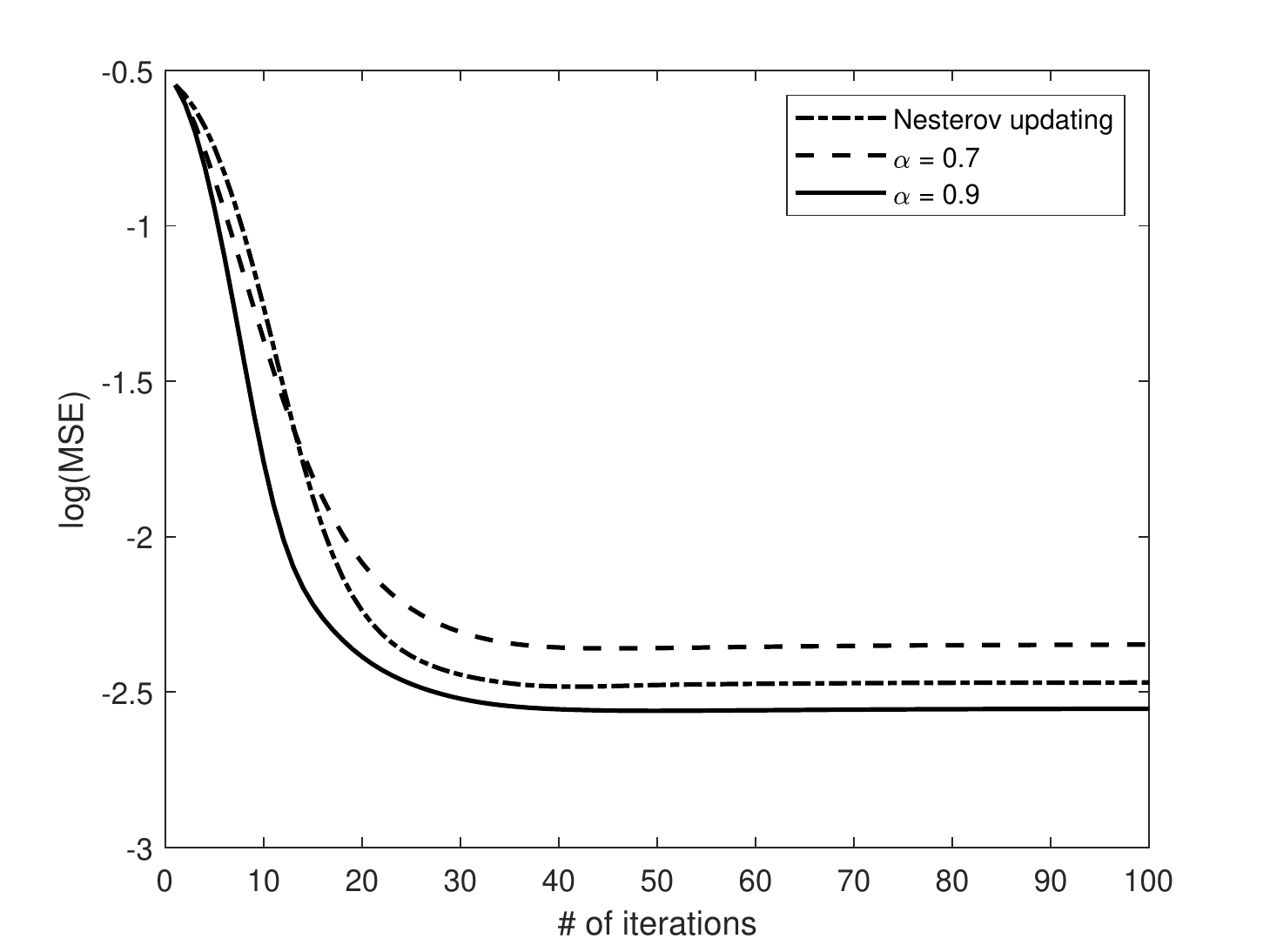}
	} 
	\subfigure[Box-plot for $K=800$]{ 
		\includegraphics[width=2.2in]{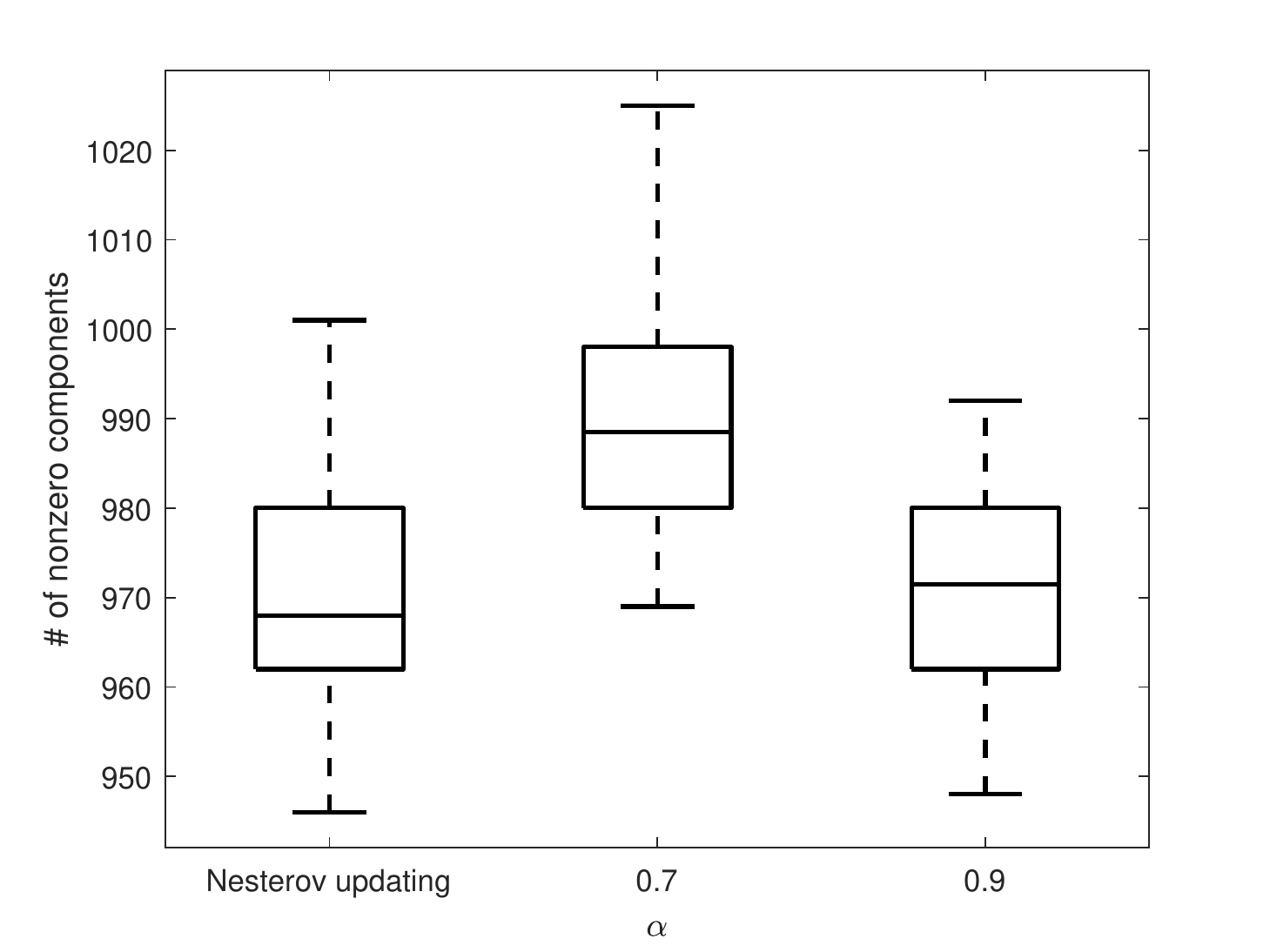} 
	} 
	\caption{The performance for setting Nesterov’s momentum coefficient and $\alpha=0.7, 0.9$.  }
	\label{fig.nesterov} 
\end{figure}


%
%


\bibliographystyle{plainnat}
\bibliography{refs.bib}   

\end{document}